\documentclass[12pt,a4paper]{article}

\usepackage{amsmath}
\usepackage{amssymb}
\usepackage{amsthm}
\usepackage{mathrsfs}
\usepackage[all]{xy}
\usepackage{graphicx}
\usepackage{bm}
\usepackage{comment}

\newtheorem{thm}{Theorem}[section]
\newtheorem{prop}[thm]{Proposition}
\newtheorem{defn}[thm]{Definition}
\newtheorem{lem}[thm]{Lemma}
\newtheorem{cor}[thm]{Corollary}

\newtheorem{rem}[thm]{Remark}

\newtheorem*{thmn}{Theorem}

\makeatletter

\@addtoreset{equation}{section}
\makeatother

\newcommand{\cf}{\textit{cf.\ }}

\DeclareMathOperator{\Gal}{Gal}
\DeclareMathOperator{\id}{id}
\DeclareMathOperator{\tr}{tr}

\DeclareMathOperator{\End}{End}
\DeclareMathOperator{\Hom}{Hom}

\DeclareMathOperator{\Spf}{Spf}

\DeclareMathOperator{\Ind}{Ind} 
\DeclareMathOperator{\Tr}{Tr}
\DeclareMathOperator{\Nr}{Nr}
\DeclareMathOperator{\Trd}{Trd}
\DeclareMathOperator{\Nrd}{Nrd}

\DeclareMathOperator{\Spa}{Spa}
\DeclareMathOperator{\Lie}{Lie}
\DeclareMathOperator{\Art}{Art}
\DeclareMathOperator{\sgn}{sgn}
\DeclareMathOperator{\pr}{pr}

\newcommand{\bom}[1]{\mbox{\boldmath $#1$}}

\newcommand{\bA}{\mathbb{A}}

\newcommand{\bF}{\mathbb{F}}

\newcommand{\bQ}{\mathbb{Q}}

\newcommand{\bZ}{\mathbb{Z}}

\newcommand{\bfC}{\mathbf{C}}

\newcommand{\bfM}{\mathbf{M}}

\newcommand{\bfg}{\mathbf{g}}

\newcommand{\cB}{\mathcal{B}}

\newcommand{\cD}{\mathcal{D}}

\newcommand{\cG}{\mathcal{G}}

\newcommand{\cL}{\mathcal{L}}
\newcommand{\cM}{\mathcal{M}}

\newcommand{\cO}{\mathcal{O}}

\newcommand{\cS}{\mathcal{S}}

\newcommand{\cX}{\mathcal{X}}

\newcommand{\fp}{\mathfrak{p}}

\newcommand{\fF}{\mathfrak{F}}

\newcommand{\fI}{\mathfrak{I}}

\newcommand{\fX}{\mathfrak{X}}

\newcommand{\iGL}{\mathit{GL}}

\newcommand{\rmab}{\mathrm{ab}}
\newcommand{\rmur}{\mathrm{ur}}
\newcommand{\rmac}{\mathrm{ac}}

\newcommand{\ol}{\overline}
\newcommand{\wh}{\widehat}

\begin{document}

\title{Affinoids in the Lubin--Tate perfectoid space and\\ simple supercuspidal representations II: wild case}
\author{Naoki Imai and Takahiro Tsushima}
\date{}
\maketitle


\begin{abstract} 
We construct a family of affinoids in the Lubin--Tate perfectoid space 
and their formal models such that 
the middle cohomology of their reductions 
realizes the 
local Langlands correspondence and the 
local Jacquet--Langlands correspondence 
for the simple supercuspidal representations. 
The reductions of the formal models are 
isomorphic to the perfections of some Artin--Schreier varieties, 
whose cohomology realizes 
primitive Galois representations. 
We show also the Tate conjecture for 
Artin--Schreier varieties associated to quadratic forms. 
\end{abstract}

\section*{Introduction}
Let $K$ be a non-archimedean local field with residue field $k$. 
Let $p$ be the characteristic of $k$. 
We write $\mathcal{O}_K$ for the ring of integers of $K$, 
and $\mathfrak{p}$ for the maximal ideal of $\mathcal{O}_K$. 
We fix an algebraic closure $k^{\mathrm{ac}}$ of $k$. 
The Lubin--Tate spaces are deformation spaces of 
the one-dimensional formal $\mathcal{O}_K$-module 
over $k^{\mathrm{ac}}$ of height $n$ with level structures. 
We take a prime number $\ell$ that is different from $p$. 
The local Langlands correspondence (LLC) and the 
local Jacquet--Langlands correspondence (LJLC) 
for supercuspidal representations of $\mathit{GL}_n$ 
are realized in the $\ell$-adic cohomology of Lubin--Tate spaces. 
This is proved in \cite{BoyMDr} and \cite{HTsimSh} 
by global automorphic arguments. 
On the other hand, 
the relation between 
these correspondences and 
the geometry of Lubin--Tate spaces is 
not well understood. 

In this direction, 
Yoshida constructs a semi-stable model of the Lubin--Tate space 
with a full level $\fp$-structure, 
and studies its relation with the LLC 
in \cite{YoLTv}. 
In this case, 
the Deligne--Lusztig varieties appear as open subschemes in 
the reductions of the semi-stable models, 
and their cohomology 
realizes the LLC for depth zero supercuspidal representations. 
In \cite{BWMax}, 
Boyarchenko--Weinstein 
construct a family of affinoids in 
the Lubin--Tate perfectoid space and their formal models 
so that 
the cohomology of the reductions 
realizes the LLC and the LJLC for 
some representations 
which are related to unramified extensions of $K$ 
(\cf \cite{WeGood} for some special case at a finite level). 
It generalizes 
a part of the result in \cite{YoLTv} 
to higher conductor cases. 
In the Lubin--Tate perfectoid setting, 
the authors study the case for 
the essentially tame simple supercuspidal representations 
in \cite{ITsimptame}, 
where simple supercuspidal means that 
the exponential Swan conductor is equal to one. 
See \cite{BHestI} for the notion of 
essentially tame  representations. 
The result in \cite{ITsimptame} 
is generalized to some higher conductor 
essentially tame cases by Tokimoto in \cite{TokLTs} 
(\cf \cite{ITLTsurf} for some special case at a finite level). 

In all the above cases, 
Langlands parameters are of the form 
$\Ind_{W_L}^{W_K} \chi$ 
for a finite separable extension $L$ over $K$ 
and a character $\chi$ of $W_L$, 
where $W_K$ and $W_L$ denote the Weil groups of 
$K$ and $L$ respectively. 
Further, the construction of affinoids 
directly involves CM points which have multiplication 
by $L$. 
In this paper, 
we study the case for 
simple supercuspidal representations 
which are not essentially tame. 
In this case, 
the Langlands parameters can not be 
written as inductions of characters. 
Hence, we have no canonical candidate of CM points 
which may be used for constructions of affinoids. 

We will explain our main result. 
All the representation are 
essentially tame if $n$ is prime to $p$. 
Hence, we assume that $p$ divides $n$. 
We say that 
a representation of $\iGL_n (K)$ is 
essentially simple supercuspidal if 
it is a character twist of a 
simple supercuspidal representation. 
Let 
$q$ be the number of the elements of $k$ and 
$D$ be the central division algebra over $K$ of invariant $1/n$. 
We write $q=p^f$ and $n=p^e n'$, where $n'$ is prime to $p$. 
We put $m =\gcd (e,f)$. 
The main theorem is the following: 
\begin{thmn}
For $r \in \mu_{q-1}(K)$, 
there is an affinoid $\cX_r$ in the Lubin--Tate perfectoid space 
and its formal model $\fX_r$ such that 
\begin{itemize}
\item 
the special fiber $\ol{\fX}_r$ of $\fX_r$ is isomorphic to 
the perfection of the affine smooth variety defined by 
\[
 z^{p^m} -z =y^{p^e +1} -\frac{1}{n'} 
 \sum_{1 \leq i \leq j \leq n-2} y_i y_j \quad 
 \textrm{in $\bA_{k^{\rmac}}^n$}, 
\]
\item 
the stabilizer $H_r \subset \iGL_n (K) \times D^{\times} \times W_K$ 
of $\cX_r$ naturally acts on $\ol{\fX}_r$, and 
\item 
$\mathrm{c\mathchar`-Ind}_{H_r}^{\iGL_n (K) \times D^{\times} \times W_K} 
 H_{\mathrm{c}}^{n-1}(\ol{\fX}_r,\ol{\bQ}_{\ell} )$ 
 realizes the LLC and the LJLC for 
 essentially simple supercuspidal representations. 
\end{itemize}
\end{thmn}
See Theorem \ref{thm:red} and Theorem \ref{Pireal} for precise statements. 
As we mentioned, 
we have no candidate of CM points 
for the construction of affinoids. 
First, we consider a CM point $\xi$ 
which has multiplication by 
a field extension of $K$ obtained by adding an $n$-th root of 
a uniformizer of $K$. 
If we imitate the construction of 
affinoids in \cite{ITsimptame} using the CM point $\xi$, 
we can get a non-trivial affinoid and its model, 
but the reduction degenerates in some sense, and 
the cohomology of the reduction 
does not give a supercuspidal representation. 
What we will do in this paper is to modify 
the CM point $\xi$ using information of 
field extensions which appear in the study of 
our simple supercuspidal Langlands parameter. 
The modified point, 
which is constructed in Proposition \ref{exxi'}, 
is no longer a CM point, 
but we can use this point for a construction of a desired affinoid. 
Since the modification comes from the study of 
the Langlands parameter, 
we expect that such constructions work also for other 
Langlands parameters. 

In the above mentioned preceding researches, 
the Langlands parameters are 
inductions of characters, and 
realized from 
commutative group actions on varieties. 
In the case for Deligne--Lusztig varieties, 
they come from the natural action of tori. 
In our simple supercuspidal case, 
they come from non-commutative group actions. 
For example, 
the restriction to the inertia subgroup 
of a simple supercuspidal Langlands parameter 
factors through a semidirect product of 
a cyclic group with a Heisenberg type group, 
which acts on our Artin--Schreier variety in a very non-trivial way. 

In the following, 
we briefly explain the content of each section. 
In Section \ref{LTps}, we collect known 
results on the Lubin--Tate perfectoid space, 
its formal model and group action on it. 

In Section \ref{GoodRedAff}, 
we construct a family of affinoids and their formal models. 
Further we determine the reductions of them.
The reduction is isomorphic to 
the perfection of some Artin--Schreier 
variety.  

In Section \ref{GroupRed}, 
we describe the group action on
the reductions. 
In Section \ref{CompAS}, we show that 
the Tate conjecture holds for 
Artin--Schreier varieties of associated to quadratic forms. 
Further, we study the action of some special element on 
cycle classes in the etale cohomology of the Artin--Schreier variety. 
This becomes a key ingredient for the proof of the main theorem. 

In Section \ref{sec:explicit}, 
we give an explicit description of 
the LLC and the LJLC for essentially simple supercuspidal representations, 
which follows from results in \cite{ITlgsw1} and \cite{ITsimpJL}. 
In Section \ref{GeomRela}, 
we give a geometric realization of 
the LLC and the LJLC in the cohomology of our reduction.

\subsection*{Acknowledgements}
The authors would like to thank a referee for 
helpful comments and suggestions. 
This work was supported by JSPS KAKENHI Grant Numbers 
26707003, 15K17506, 18H01109, 20K03529. 

\subsection*{Notation}
For a non-archimedean valuation field $F$, 
its valuation ring is denoted by $\mathcal{O}_F$. 
For a non-archimedean valuation field $F$ 
and an element $a \in \mathcal{O}_F$, 
its image in the residue field is denoted by $\bar{a}$. 
For $a \in \mathbb{Q}$ and 
elements $f$, $g$ with valuation $v$ that takes 
values in $\mathbb{Q}$, 
we write 
$f \equiv g \mod\!_{\geq}\, a$ 
if $v(f-g) \geq a$, and  
$f \equiv g  \mod\!_{>}\, a$ 
if $v(f-g) >a$. 
For a topological field extension $E$ over $F$, 
let $\Gal (E/F)$ denote the group of 
the continuous automorphisms of $E$ over $F$. 
For an ideal $I$ of a topological ring, 
let $I^-$ denote the closure of $I$. 

\section{Lubin--Tate perfectoid space}\label{LTps}
\subsection{Lubin--Tate perfectoid space and its formal model}\label{LTpsfm}
Let $K$ be a non-archimedean local field with 
residue field $k$ of characteristic $p$. 
Let $q$ be the number of the elements of $k$. 
We write $\mathfrak{p}$ for the maximal ideal of $\mathcal{O}_K$. 
We fix an algebraic closure 
$K^{\mathrm{ac}}$ of $K$. 
Let 
$k^{\mathrm{ac}}$ be the residue field of 
$K^{\mathrm{ac}}$. 

Let $n$ be a positive integer. 
We take a one-dimensional formal 
$\mathcal{O}_K$-module $\cG_0$ over $k^{\mathrm{ac}}$ 
of height $n$, 
which is unique up to isomorphism. 
Let $K^{\mathrm{ur}}$
be the maximal unramified extension of $K$ in $K^{\mathrm{ac}}$.
We write $\widehat{K}^{\mathrm{ur}}$ for the completion of $K^{\mathrm{ur}}$. 
Let $\{ \Spf A_m \}_{m \geq 0}$ 
be the tower of Lubin-Tate formal schemes defined by 
Drinfeld level $\mathfrak{p}^m$-structure 
as explained in \cite[\S 1.1]{ITsimptame}. 
Note that the generic fibers of these formal schemes 
are connected components of usual Lubin-Tate spaces. 
Let $I$ the ideal of 
$\varinjlim A_m$ generated by 
the maximal ideal of $A_0$. 
Let $A$ be the $I$-adic 
completion of $\varinjlim A_m$. 
We put $\bfM_{\cG_0, \infty}=\Spf A$. 

Let $K^{\mathrm{ab}}$ be the maximal abelian 
extension of $K$ in $K^{\mathrm{ac}}$. 
We write 
$\widehat{K}^{\mathrm{ab}}$ for the completion of 
$K^{\mathrm{ab}}$. 
Let $\wedge \cG_0$ 
denote the one-dimensional formal 
$\mathcal{O}_K$-module over $k^{\mathrm{ac}}$ of height one. 
Then we have 
$\bfM_{ \wedge \cG_0, \infty} \simeq 
 \Spf \mathcal{O}_{\widehat{K}^{\mathrm{ab}}}$ 
by the Lubin--Tate theory. 
We have a determinant morphism 
\begin{equation}\label{mor}
 \bfM_{\cG_0, \infty} \to 
 \bfM_{\wedge \cG_0, \infty} 
\end{equation} 
by \cite[2.5 and 2.7]{WeSemi} 
(\cf \cite{HedPhD}). 
Then, we have the ring homomorphism 
$\cO_{\wh{K}^{\rmab}} \to A$ 
determined by \eqref{mor}. 

We fix a uniformizer 
$\varpi$ of $K$. 
Let 
$\mathcal{M}_{\infty}$ 
be the open adic subspace of 
$\Spa (A,A)$ defined by 
$|\varpi(x)| \neq 0$ 
(\cf \cite[2]{Hugfs}). 
We regard $\mathcal{M}_{\infty}$ 
as an adic space over $\wh{K}^{\rmur}$. 
Let $\mathbf{C}$ be the completion of $K^{\mathrm{ac}}$. 
For a deformation $\cG$ of $\cG_0$ over $\cO_{\bfC}$, 
we put 
\[
 V_{\fp} (\cG) =\bigl( \varprojlim \cG(\cO_{\bfC})[\fp^m] \bigr) 
 \otimes_{\cO_K} K, 
\]
where 
$\cG(\cO_{\bfC})[\fp^m]$ denotes the $\cO_K$-module of 
the $\fp^m$-torsion points of 
$\cG(\cO_{\bfC})$ and 
the transition maps are multiplications by $\varpi$. 
By the construction, each point of 
$\cM_{\infty}(\bfC)$ 
corresponds to a triple $(\cG,\phi,\iota)$ 
that consists of 
a formal $\cO_K$-module $\cG$ over $\cO_{\bfC}$, 
an isomorphism $\phi \colon K^n \to V_{\fp} (\cG)$ and 
an isomorphism 
$\iota \colon \cG_0 \to \cG \otimes_{\cO_{\bfC}} k^{\rmac}$ 
(\cf \cite[Definition 2.10.1]{BWMax}). 

We put 
$\eta=\Spa (\wh{K}^{\mathrm{ab}}, 
 \cO_{\wh{K}^{\mathrm{ab}}})$. 
By the ring homomorphism 
$\cO_{\wh{K}^{\mathrm{ab}}} \to A$, 
we can regard $\cM_{\infty}$ as an 
adic space over $\eta$, 
for which we write 
$\mathcal{M}_{\infty, \eta}$. 
We put 
$\bar{\eta}=\Spa (\mathbf{C}, \mathcal{O}_{\mathbf{C}})$ 
and 
$\mathcal{M}_{\infty, \overline{\eta}} = 
 \mathcal{M}_{\infty, \eta} \times_{\eta} \ol{\eta}$. 
Then, 
$\mathcal{M}_{\infty, \overline{\eta}}$ 
is a perfectoid space over $\mathbf{C}$ 
in the sense of 
\cite[Definition 6.15]{SchPerf} 
by \cite[Lemma 2.32]{WeSemi}. 
We call 
$\mathcal{M}_{\infty, \overline{\eta}}$ 
the Lubin--Tate perfectoid space. 

In the following, 
we recall an explicit description of 
$A^{\circ}=A \widehat{\otimes}_{\mathcal{O}_{\widehat{K}^{\mathrm{ab}}}} 
 \mathcal{O}_{\mathbf{C}}$ 
given in \cite[(2.8)]{WeSemi}. 
Let $\wh{\cG}_0$ be the formal $\cO_K$-module 
over $\cO_K$ whose logarithm is 
\[
 \sum_{i=0}^{\infty} \frac{X^{q^{in}}}{\varpi^i} 
\]
(\cf \cite[2.3]{BWMax}). 
Let $\cG_0$ be the formal $\cO_K$-module over 
$k^{\rmac}$ obtained as the reduction of $\wh{\cG}_0$. 
We put $\mathcal{O}_D =\End \cG_0$ and 
$D=\mathcal{O}_D \otimes_{\cO_K} K$, which is 
the central division algebra over $K$ of invariant $1/n$. 
Let $[\ \cdot \ ]$ denote the action of $\mathcal{O}_D$ on $\cG_0$. 
Let $\varphi$ be the element of $D$ such that $[\varphi](X)=X^q$. 
Let $K_n$ 
be the unramified extension of $K$ of degree $n$. 
We consider the $K$-algebra embedding of $K_n$ 
into $D$ determined by 
\[
 [\zeta](X)=\bar{\zeta} X \ \ 
 \text{for} \ \ \zeta \in \mu_{q^n -1}(K_n). 
\]
Then we have $\varphi^n=\varpi$ and 
$\varphi \zeta =\zeta^q \varphi$ for $\zeta \in \mu_{q^n -1} (K_n)$. 
Let $\wh{\wedge \cG_0}$ 
be the one-dimensional formal $\mathcal{O}_K$-module 
over $\cO_K$ whose logarithm is 
\[
 \sum_{i=0}^{\infty} (-1)^{(n-1)i} \frac{X^{q^i}}{\varpi^i}. 
\]
We choose a compatible system 
$\{t_m\}_{m \geq 1}$ such that 
\begin{equation}\label{tmcho}
 t_m \in {K}^{\mathrm{ac}} \quad (m \geq 1), \quad 
 t_1 \neq 0, \quad 
 [\varpi]_{\wh{\wedge \cG_0}}(t_1)=0, \quad 
 [\varpi]_{\wh{\wedge \cG_0}}(t_m)=t_{m-1} \quad (m \geq 2).
\end{equation}
We put
\[
 t=\lim_{m \to \infty}(-1)^{q(n-1)(m-1)} t_m^{q^{m-1}} 
 \in \mathcal{O}_{\mathbf{C}}.
\]
Let $v$ be the normalized valuation of $K$ 
such that $v(\varpi)=1$. 
The valuation $v$ naturally extends to a valuation on 
$\mathbf{C}$, for which we again write $v$. 
Note that $v(t)=1/(q-1)$. 
For an integer $i \geq 0$, we put 
\[
 t^{q^{-i}} =\lim_{m \to \infty} (-1)^{q(n-1)(m-1)} t_m^{q^{m-i-1}}. 
\]

Let $W_K$ be the 
Weil group of $K$. 
Let 
$\mathrm{Art}_K 
\colon K^{\times} \xrightarrow{\sim} W_K^{\mathrm{ab}}$
be the Artin 
reciprocity map normalized 
such that a uniformizer is sent to a
lift of the geometric Frobenius element. 
We use similar normalizations also for the Artin reciprocity maps 
for other non-archimedean local fields. 
Let $\sigma \in W_{K}$. 
Let $n_{\sigma}$ 
be the image of $\sigma$ under the composite 
\[
 W_K \twoheadrightarrow W_K^{\mathrm{ab}} 
 \xrightarrow{\mathrm{Art}^{-1}_K}
 K^{\times} \xrightarrow{v} \mathbb{Z}. 
\]
Let $a_K \colon W_K \to \mathcal{O}_K^{\times}$ 
be the homomorphism given by the action of $W_K$ 
on $\{t_m\}_{m \geq 1}$. 
It induces an isomorphism 
$a_K \colon 
\mathrm{Gal}(\widehat{K}^{\mathrm{ab}}/\widehat{K}^{\mathrm{ur}}) \simeq 
\mathcal{O}_K^{\times}$. 

For $m \geq 0$, we put 
\begin{equation}\label{deltam}
 \delta_m (X_1,\ldots,X_n)
 =\widehat{\wedge \cG_0}\sum_{(m_1,\ldots ,m_n)}
 \sgn (m_1,\ldots ,m_n) 
 X_1^{q^{m_1-m}} \cdots X_n^{q^{m_n-m}} 
\end{equation}
in 
$\cO_K [[X_1^{1/q^{\infty}},\ldots, X_n^{1/q^{\infty}}]]$, 
where 
\begin{itemize}
\item 
the symbol $\widehat{\wedge \cG_0}\sum$ 
denotes the sum under the additive operation of 
$\widehat{\wedge \cG_0}$, 
\item 
we take the sum over $n$-tuples 
$(m_1,\ldots ,m_n)$ of integers 
such that 
$m_1 + \cdots + m_n =n(n-1)/2$ and 
$m_i \not\equiv m_j \mod n$ for $i \neq j$, 
\item 
$\sgn (m_1,\ldots ,m_n)$ 
is the sign of the permutation on $\bZ/n\bZ$ 
defined by 
$i \mapsto m_{i+1}$. 
\end{itemize}
We put 
\[
 \delta=\lim_{m \to \infty} \delta_m^{q^m} \in 
 \cO_{\mathbf{C}}[[X_1^{1/q^{\infty}},\ldots, X_n^{1/q^{\infty}}]]. 
\] 
For $l \geq 1$, 
we put 
\[
 \delta^{q^{-l}}=\lim_{m \to \infty} \delta_m^{q^{m-l}}. 
\]
The following theorem follows from 
\cite[(2.8)]{WeSemi} 
and the proof of \cite[Theorem 2.10.3]{BWMax} 
(\cf \cite[Theorem 6.4.1]{ScWeMpd}). 

\begin{thm}[{\cite[Theorem 1.3]{ITsimptame}}]
Let 
$\sigma \in \mathrm{Gal}
 (\widehat{K}^{\mathrm{ab}}/\widehat{K}^{\mathrm{ur}})$. 
We put 
$A^{\sigma}=A 
 \widehat{\otimes}_{\mathcal{O}_{\widehat{K}^{\mathrm{ab}}}, \sigma} 
 \mathcal{O}_{\mathbf{C}}$.
Then, we have an isomorphism 
\begin{equation}\label{Astr}
 A^{\sigma} 
 \simeq 
 \mathcal{O}_{\mathbf{C}}[[X_1^{1/q^{\infty}},\ldots,X_n^{1/q^{\infty}}]]/
 (\delta(X_1,\ldots,X_n)^{q^{-m}} - \sigma(t^{q^{-m}}))_{m \geq 0}^-. 
\end{equation}
\end{thm}

For 
$\sigma \in \Gal (\widehat{K}^{\mathrm{ab}}/\widehat{K}^{\mathrm{ur}})$, 
let $\mathcal{M}_{\infty,\bar{\eta},\sigma}$ 
be the base change of $\cM_{\infty,\eta}$ 
by $\bar{\eta} \to \eta \xrightarrow{\sigma} \eta$. 
For 
$\sigma \in \Gal (\widehat{K}^{\mathrm{ab}}/\widehat{K}^{\mathrm{ur}})$ 
and 
$\alpha =a_K (\sigma) \in \mathcal{O}_K^{\times}$, 
we write $A^{\alpha}$ for $A^{\sigma}$ and 
$\mathcal{M}^{(0)}_{\infty,\bar{\eta},\alpha}$ for 
$\mathcal{M}^{(0)}_{\infty,\bar{\eta},\sigma}$. 
We put 
\begin{equation}\label{Mdec}
 \bfM_{\infty,\mathcal{O}_{\mathbf{C}}}^{(0)} 
 =\coprod_{\alpha \in \mathcal{O}_K^{\times}} 
 \Spf A^{\alpha}, \quad 
 \mathcal{M}^{(0)}_{\infty,\bar{\eta}} 
 =\coprod_{\alpha \in \mathcal{O}_K^{\times}} 
 \mathcal{M}_{\infty,\bar{\eta},\alpha}. 
\end{equation}
Then $\mathcal{M}^{(0)}_{\infty,\bar{\eta}}$ 
is the generic fiber of 
$\bfM_{\infty,\mathcal{O}_{\mathbf{C}}}^{(0)}$, 
and 
$\mathcal{M}^{(0)}_{\infty,\bar{\eta}} (\bfC) = \cM_{\infty}(\bfC)$. 

Let $+_{\wh{\cG_0}}$ and 
$+_{\wh{\wedge \cG_0}}$ be the additive operations for 
$\wh{\cG_0}$ and $\wh{\wedge \cG_0}$ respectively. 

\begin{lem}[{\cite[Lemma 1.5]{ITsimptame}}]\label{appsum}
\begin{enumerate}
\item
We have $X_1 +_{\wh{\cG_0}} X_2 \equiv X_1 + X_2$ 
modulo terms of total degree $q^n$. 
\item
We have $X_1 +_{\wh{\wedge \cG_0}} X_2 \equiv X_1 + X_2$ 
modulo terms of total degree $q$. 
\end{enumerate}
\end{lem}

Let $\bom{X}\!_i$ be $(X_i^{q^{-j}} )_{j \geq 0}$ for $1 \leq i \leq n$. 
We write 
$\delta(\bom{X}\!_1, \ldots,\bom{X}\!_n)$ for 
the $q$-th power compatible system 
$(\delta(X_1,\ldots,X_n)^{q^{-j}} )_{j \geq 0}$. 

For $q$-th power compatible systems 
$\bom{X}=(X^{q^{-j}} )_{j \geq 0}$ and 
$\bom{Y}=(Y^{q^{-j}} )_{j \geq 0}$ that take values in $\cO_{\bfC}$, 
we define $q$-th power compatible systems 
$\bom{X} +\bom{Y}$, $\bom{X} -\bom{Y}$ and $\bom{X} \bom{Y}$ 
by the requirement that their $j$-th components for $j \geq 0$ are 
\[
 \lim_{m \to \infty} (X^{q^{-m}} +Y^{q^{-m}})^{q^{m-j}}, \quad 
 \lim_{m \to \infty} (X^{q^{-m}} -Y^{q^{-m}})^{q^{m-j}}, \quad 
 \textrm{and} \quad 
 X^{q^{-j}} Y^{q^{-j}}
\]
respectively. 
For such $\bom{X}=(X^{q^{-j}} )_{j \geq 0}$, 
we put $v(\bom{X}) = v(X)$. 
We put 
\[
 \delta'_0 (\bom{X}\!_1, \ldots,\bom{X}\!_n) 
 =\sum_{(m_1,\ldots ,m_n)}
 \sgn (m_1,\ldots ,m_n) 
 \bom{X}\!_1^{\,q^{m_1}} \cdots \bom{X}\!_n^{\,q^{m_n}} , 
\]
where 
we take the sum in the above sense and 
the index set is the same as \eqref{deltam}. 

\begin{lem}[{\cite[Lemma 1.6]{ITsimptame}}]\label{mixed}
Assume that $n \geq 2$ and 
$v(\bom{X}\!_i) \geq (n q^{i-1}(q-1))^{-1}$ for $1 \leq i \leq n$. 
Then, we have 
\[
 \delta(\bom{X}\!_1, \ldots,\bom{X}\!_n) \equiv 
 \delta'_0 (\bom{X}\!_1,\ldots,\bom{X}\!_n) \mod\!_>\, 
 \frac{1}{n}+\frac{1}{q-1}. 
\]
\end{lem}

\subsection{Group action on the formal model}\label{gpfm}
We define a group action on the formal scheme 
$\bfM_{\infty,\mathcal{O}_{\mathbf{C}}}^{(0)}$, 
which is compatible with usual group actions on 
Lubin--Tate spaces with finite level (\cf \cite[2.11]{BWMax}). 
We put 
\[
 G=\mathit{GL}_n(K) \times D^{\times} \times W_K.
\] 
Let $G^0$ denote the kernel of the following homomorphism:
\[
 G \to \mathbb{Z} ;\ (g,d,\sigma) \mapsto v 
 \bigl( \det(g)^{-1} 
 \mathrm{Nrd}_{D/K}(d)
 \mathrm{Art}^{-1}_K(\sigma) \bigr). 
\]
Then, the formal scheme 
$\bfM_{\infty, \mathcal{O}_{\mathbf{C}}}^{(0)}$ 
admits a right action of $G^0$. 
We write down the action. 
In the sequel, we use the following notation: 
\begin{quote}
For $a \in \mu_{q^n -1} (K_n) \cup \{0\}$, 
let $a^{q^{-m}}$ denote the 
$q^m$-th root of $a$ in $\mu_{q^n -1} (K_n) \cup \{0\}$ 
for a positive integer $m$, and 
we simply write $a$ also for the   
$q$-th power compatible system $(a^{q^{-m}})_{m \geq 0}$. 
\end{quote}

For $q$-th power compatible systems 
$\bom{X}=(X^{q^{-j}} )_{j \geq 0}$ and 
$\bom{Y}=(Y^{q^{-j}} )_{j \geq 0}$ that take values in $\cO_{\bfC}$, 
we define a $q$-th power compatible system 
$\bom{X} +_{\wh{\cG_0}} \bom{Y}$ 
by the requirement that their $j$-th components for $j \geq 0$ are 
\[
 \lim_{m \to \infty} (X^{q^{-m}} +_{\wh{\cG_0}} Y^{q^{-m}})^{q^{m-j}}. 
\]
The symbol $\wh{\cG_0} \sum$ 
denotes this summation for $q$-th power compatible systems. 

First, we define 
a left action of $\mathit{GL}_n(K) \times D^{\times}$
on the ring 
\[
 B_n=\mathcal{O}_{\mathbf{C}} 
 [[X_1^{1/q^{\infty}},\ldots,X_n^{1/q^{\infty}}]]. 
\]
For 
$a=\sum_{j=l}^{\infty}a_j\varpi^{j} \in K$ with 
$l \in \mathbb{Z}$ and 
$a_j \in \mu_{q-1} (K) \cup \{0\}$, 
we set
\[
 [a] \cdot \bom{X}\!_i = 
 \wh{\cG_0} \sum_{j=l}^{\infty} a_j \bom{X}\!_i^{\,q^{jn}}. 
\]
for $1 \leq i \leq n$. 
Let $g \in \mathit{GL}_n(K)$. 
We write 
$g =(a_{i,j})_{1 \leq i,j \leq n}$.
Then, let $g$ act on the ring $B_n$
by 
\begin{equation}\label{gl}
 g^{\ast} \colon B_n \to B_n;\ \bom{X}\!_i \mapsto 
 \wh{\cG_0} \sum_{j=1}^n [a_{j,i}] \cdot \bom{X}\!_j 
 \quad \textrm{for $1 \leq i \leq n$}.
\end{equation}
Let $d \in D^{\times}$. 
We write 
$d^{-1} =\sum_{j=l}^{\infty} d_j \varphi^j \in D^{\times}$ 
with 
$l \in \mathbb{Z}$ 
and $d_j \in \mu_{q^n -1} (K_n) \cup \{0\}$. 
Then, let $d$ act on $B_n$ by 
\begin{equation}\label{divi}
 d^{\ast} \colon B_n \to B_n ;\ 
 \bom{X}\!_i \mapsto 
 \wh{\cG_0} \sum_{j=l}^{\infty } d_j \bom{X}\!_i^{\,q^{j}} 
 \quad \textrm{for $1 \leq i \leq n$}.
\end{equation}
Now, we give a right action of $G^0$ on 
$\bfM_{\infty, \mathcal{O}_{\mathbf{C}}}^{(0)}$ 
using \eqref{gl} and \eqref{divi}. 
Let $(g,d,1) \in G^0$. 
We set 
\[
 \gamma(g,d)=
 \det (g) \mathrm{Nrd}_{D/K}(d)^{-1} \in 
 \mathcal{O}_K^{\times}. 
\]
We put $\bom{t} =(t^{q^{-m}})_{m \geq 0}$. 
Let $(g,d,1)$ act on 
$\bfM_{\infty, \mathcal{O}_{\mathbf{C}}}^{(0)}$ 
by 
\[
 A^{\alpha} \to A^{\gamma(g,d)^{-1} \alpha };\ 
 \bom{X}\!_i \mapsto (g,d) \cdot \bom{X}\!_i 
 \quad \textrm{for $1 \leq i \leq n$}, 
\]
where $\alpha \in \cO_K^{\times}$. 
This is well-defined, because 
the equation 
\[
 \delta ((g,d) \cdot \bom{X}\!_1 , \ldots, (g,d) \cdot \bom{X}\!_n )=
 \Art_K (\alpha) (\bom{t}) 
\]
is equivalent to 
$\delta (\bom{X}\!_1 , \ldots, \bom{X}\!_n )=
 \Art_K (\gamma(g,d)^{-1} \alpha) (\bom{t})$. 
Let $(1,\varphi^{-n_{\sigma}},\sigma) \in G^0$ act on 
$\bfM_{\infty, \mathcal{O}_{\mathbf{C}}}^{(0)}$ by 
\[
 A^{\alpha} \to A^{a_K(\sigma)\alpha};\ \bom{X}\!_i \mapsto 
 \bom{X}\!_i, \hspace{1.0em} x \mapsto 
 \sigma (x) \quad 
 \textrm{ for $1 \leq i \leq n$ and $x \in \mathcal{O}_{\mathbf{C}}$}, 
\]
where $\alpha \in \cO_K^{\times}$. 
Thus, we have a right action of $G^0$ on 
$\bfM_{\infty,\mathcal{O}_{\mathbf{C}}}$, 
which induces a right action on 
$\mathcal{M}^{(0)}_{\infty,\bar{\eta}} (\bfC) = \cM_{\infty}(\bfC)$. 

\begin{rem}\label{Ktri}
For $a \in K^{\times}$, 
the action of $(a,a,1) \in G^{0}$ 
on $\bfM_{\infty,\mathcal{O}_{\mathbf{C}}}$ 
is trivial by the definition. 
\end{rem}

\subsection{CM points}\label{CMpts}
We recall the notion of CM points from \cite[3.1]{BWMax}. 
Let $L$ be a finite extension of $K$ of degree $n$ 
inside $\bfC$. 

\begin{defn}
A deformation $\cG$ of $\cG_0$ over $\cO_{\bfC}$ 
has CM by $L$ if 
there is an isomorphism 
$L \xrightarrow{\sim} \End (\cG) \otimes_{\cO_K} K$ 
as $K$-algebras 
such that the induced map 
$L \to \End (\Lie \cG) \otimes_{\cO_K} K \simeq \bfC$ 
coincides with the natural embedding $L \subset \bfC$. 
\end{defn}

We say that a point of $\cM_{\infty}(\bfC)$ 
has CM by $L$ if the corresponding deformation over $\cO_{\bfC}$ 
has CM by $L$. 

Let $\xi \in \cM_{\infty}(\bfC)$ be a 
point that has CM by $L$. 
Let $(\cG,\phi,\iota)$ be the triple corresponding to $\xi$. 
Then we have embeddings 
$i_{M,\xi} \colon L \to M_n (K)$ 
and 
$i_{D,\xi} \colon L \to D$ 
characterized by the commutative diagrams 
\[
 \xymatrix{
 K^n \ar@{->}^-{\phi}[r] \ar@{->}_-{i_{M,\xi} (a)}[d] & 
 V_{\fp} (\cG)  
 \ar@{->}^-{V_{\fp} (a)}[d] \\ 
 K^n \ar@{->}^-{\phi}[r] & 
 V_{\fp} (\cG) 
 }
 \quad \quad \lower20pt\hbox{\textrm{and}} \quad \quad 
 \xymatrix{
 \cG_0 \ar@{->}^-{\iota}[r] \ar@{->}_-{i_{D,\xi} (a)}[d] & 
 \cG \otimes_{\cO_{\bfC}} k^{\rmac}  
 \ar@{->}^-{a \otimes \id }[d] \\ 
 \cG_0 \ar@{->}^-{\iota}[r] & 
 \cG \otimes_{\cO_{\bfC}} k^{\rmac} 
 }
\]
in the isogeny category for $a \in L$. 
We put 
$i_{\xi}=(i_{M,\xi} ,i_{D,\xi}) \colon L \to M_n (K) \times D$. 
We put 
\[
 (\iGL_n (K) \times D^{\times})^0 =\{ (g,d) \in 
 \iGL_n (K) \times D^{\times} \mid (g,d,1) \in G^0 \}. 
\]

\begin{lem}[{\cite[Lemma 3.1.2]{BWMax}}]\label{CMtrans}
The group $(\iGL_n (K) \times D^{\times})^0$ 
acts transitively on the set of the points of 
$\cM_{\infty}(\bfC)$ 
that have CM by $L$. 
For $\xi \in \cM_{\infty}(\bfC)$ 
that has CM by $L$, 
the stabilizer of $\xi$ in $(\iGL_n (K) \times D^{\times})^0$ 
is $i_{\xi} (L^{\times})$. 
\end{lem}

\section{Good reduction of affinoids}\label{GoodRedAff}
\subsection{Construction of affinoids}\label{ConstAff}
We take a uniformizer $\varpi$ of $K$. 
Let $r \in \mu_{q-1} (K)$. 
We put $\varpi_r =r\varpi$. 
We take 
$\varphi_r \in \bfC$ such that 
$\varphi_r^n = \varpi_r$. 
We apply results in Section \ref{LTps} 
replacing $\varpi$ with $\varpi_r$. 
We put $L_r =K(\varphi_r)$. 
By the $\cO_K$-algebra embedding 
$\cO_{L_r} \to \cO_D$ defined by 
$\varphi_r \mapsto \varphi$, 
we view $\cG_0$ as a formal $\cO_{L_r}$-module of height $1$. 
Let $\cG_r$ be a lift of $\cG_0$ to $\cO_{\wh{L}_r^{\rmur}}$ 
as a formal $\cO_{L_r}$-module. 
We take a compatible system $\{t_{r,m}\}_{m \geq 1}$ 
in $\bfC$ such that 
\[
 t_{r,1} \neq 0, \quad 
 [\varphi_r]_{\cG_r}(t_{r,1})=0, \quad 
 [\varphi_r]_{\cG_r}(t_{r,m})=t_{r,m-1} 
\]
for $m \geq 2$. 
We put 
\[
 \varphi_{M,r} = 
 \begin{pmatrix}
 \bm{0} & I_{n-1} \\
 \varpi_r & \bm{0} \\
 \end{pmatrix}
 \in M_n (K) 
\]
and $\varphi_{D,r} =\varphi \in D$. 
For $\xi \in \cM^{(0)}_{\infty,\ol{\eta}}(\bfC)$, 
we write $(\bom{\xi}_1,\ldots,\bom{\xi}_n)$ 
for the coordinate of $\xi$ 
with respect to $(\bom{X}\!_1,\ldots,\bom{X}\!_n)$, 
where $\bom{\xi}_i =(\xi_{i}^{q^{-j}})_{j \geq 0}$ for 
$1 \leq i \leq n$. 

\begin{lem}\label{exxi}
There exists 
$\xi_r \in \cM^{(0)}_{\infty,\ol{\eta}}(\bfC)$ 
such that 
\begin{equation*}\label{xilim}
 \xi_{r,i}^{q^{-j}} =\lim_{m \to \infty}  
 t_{r,m}^{q^{m-i-j}} \in \mathcal{O}_{\mathbf{C}} 
\end{equation*}
for $1 \leq i \leq n$ and $j \geq 0$. 
Further, we have the following: 
\begin{enumerate}
\item 
$\xi_r$ has CM by $L_r$. 
\item 
We have 
$i_{\xi_{r}} (\varphi_r) = 
 (\varphi_{M,r} ,\varphi_{D,r} ) \in M_n (K) \times D$. 
\item 
$\bom{\xi}_{r,i} = \bom{\xi}_{r,i+1}^q$ 
for $1 \leq i \leq n-1$. 
\item 
$v(\xi_{r,i})=1/(n q^{i-1}(q-1))$ for $1 \leq i \leq n$. 
\end{enumerate}
\end{lem}
\begin{proof}
This is proved in the same way as 
\cite[Lemma 2.2]{ITsimptame}. 
\end{proof}

We take $\xi_r$ as in Lemma \ref{exxi}. 
We can replace the choice of \eqref{tmcho} 
so that $\delta (\bom{\xi}_1,\ldots,\bom{\xi}_n) =\bom{t}$. 
Then we have 
$\xi_r \in \mathcal{M}^{(0)}_{\infty,\bar{\eta},1}$. 
Let 
$\cD_{\bfC}^{n,\mathrm{perf}}$ 
be the generic fiber of 
$\Spf \mathcal{O}_{\mathbf{C}}
 [[X_1^{1/q^{\infty}},\ldots,X_n^{1/q^{\infty}}]]$. 
We consider 
$\cM^{(0)}_{\infty,\ol{\eta},1}$ as a subspace of 
$\cD_{\bfC}^{n,\mathrm{perf}}$ 
by \eqref{Astr}. 
We put $\bom{\eta}_r =\bom{\xi}_{r,1}^{q-1}$ 
and write $\bom{\eta}_r =(\eta_r^{q^{-j}})_{j \geq 0}$. 
Note that $v(\bom{\eta}_r)=1/n$. 
We write $n=p^e n'$ with $\gcd (p,n')=1$. 
We assume that $e \geq 1$ in the sequel, 
since the case where $e=0$ is already studied in \cite{ITsimptame}. 
We put 
\[
 \varepsilon_0 = 
 \begin{cases}
 (n' +1)/2 & \textrm{if $p^e =2$,} \\ 
 0 & \textrm{if $p^e \neq 2$.} 
 \end{cases}
\]
We take $q$-th power compatible systems 
$\bom{\theta}_r =(\theta_r^{q^{-j}})_{j \geq 0}$ and 
$\bom{\lambda}_r =(\lambda_r^{q^{-j}})_{j \geq 0}$ 
in $\bfC$ satisfying 
\begin{equation}\label{defthelam}
 \bom{\theta}_r^{p^{2e}}+\bom{\eta}_r^{p^e -1} 
 (\bom{\theta}_r +1 ) =0, \quad 
 \bom{\lambda}_r^q - \bom{\eta}_r^{q -1} 
 (\bom{\lambda}_r - \bom{\theta}_r^{p^e} (\bom{\theta}_r +1 ) 
 + \varepsilon_0 \bom{\eta}_r )=0 . 
\end{equation}
Note that 
\begin{equation*}\label{valtl}
 v(\bom{\theta}_r) = \frac{p^e-1}{np^{2e}}, \quad 
 v(\bom{\lambda}_r) = \frac{1}{n} \biggl( 1-\frac{1}{qp^e} \biggr). 
\end{equation*}
We define 
$\xi'_r \in \cD_{\bfC}^{n,\mathrm{perf}}$ by 
\begin{align*}
 & \bom{\xi}'_{r,1} =\bom{\xi}_{r,1} (1+\bom{\theta}_r), \quad 
 \bom{\xi}'_{r,i+1} ={\bom{\xi}'}^{\frac{1}{q}}_{\!\!r,i} 
 \quad \textrm{for $1 \leq i \leq n-2$}, \\ 
 &\bom{\xi}'_{r,n} = 
 {\bom{\xi}'}^{\frac{1}{q}}_{\!\!r,n-1} \bigl( 
 (1+ \bom{\theta}_r)^{-n} (1 +n'\bom{\lambda}_r) 
 \bigr)^{\frac{1}{q^{n-1}}}. 
\end{align*}
\begin{prop}\label{exxi'}
There uniquely exists 
$\xi^0_r \in \cM^{(0)}_{\infty,\ol{\eta},1}$ satisfying 
\begin{align*}
 \bom{\xi}^0_{r,i} =\bom{\xi}'_{r,i} 
 \quad \textrm{for $1 \leq i \leq n-1$}, \quad 
 \bom{\xi}^0_{r,n} \equiv 
 \bom{\xi}'_{r,n} 
 \mod\!_>\, \frac{q^2 -q +1}{nq^{n-1}(q-1)} . 
\end{align*}
\end{prop}
\begin{proof}
We have 
\[
 \delta (\bom{\xi}'_r) \equiv \bom{t} 
 \mod\!_>\, \frac{1}{q-1} + \frac{1}{n}. 
\]
Hence, we see the claim by Newton's method. 
\end{proof}

\begin{rem}
The key ingredients for the construction of $\xi^0_r$ 
are the elements 
$\bom{\theta}_r$ and $\bom{\lambda}_r$ defined by \eqref{defthelam}. 
Up to some difference of normalizations, 
these elements are analogues of 
$\beta_{\zeta}$ and $\gamma_{\zeta}$ in 
\cite[\S 2.2]{ITlgsw1}, which are generators of a field extension 
used in a construction of a Langlands parameter there. 
\end{rem}

We take 
$\xi^0_r$ as in Proposition \ref{exxi'}. 
We put $\bom{x}_i =\bom{X}\!_i /\bom{\xi}^0_{r,i}$ for 
$1 \leq i \leq n$. 
We define 
$\cX_r \subset \cM^{(0)}_{\infty,\ol{\eta},1}$ by 
\begin{equation}\label{defcX}
\begin{split}
 &v \biggl( \frac{\bom{x}_i}{\bom{x}_{i+1}} -
 \Bigl( \frac{\bom{x}_{n-1}}{\bom{x}_n} \Bigr)^{q^{n-1-i}} 
 \biggr) 
 \geq \frac{1}{2nq^i} \quad 
 \textrm{for $1 \leq i \leq n-2$}, \\ 
 &v(\bom{x}_i -1) \geq \frac{1}{nq^{n-1}(p^e +1)} \quad 
 \textrm{for $n-1 \leq i \leq n$}. 
\end{split}
\end{equation} 
The definition 
of $\cX_r$ is independent of 
the choice of $\bom{\theta}_r$ and $\bom{\lambda}_r$. 
We define 
$\cB_r \subset \cD_{\bfC}^{n,\mathrm{perf}}$ 
by the same condition \eqref{defcX}.

\subsection{Formal models of affinoids}\label{RedAff}
Let $(\bom{X}\!_1,\ldots,\bom{X}\!_n)$ be the coordinate of $\cB_r$. 
We put 
$h (\bom{X}\!_1 ,\ldots ,\bom{X}\!_n )=\prod_{i=1}^n \bom{X}\!_i^{\,q^{i-1}}$. 
Further, we put 
\begin{align}
 f(\bom{X}\!_1, \ldots, \bom{X}\!_n) &=1-\frac{\delta (\bom{X}\!_1, \ldots, \bom{X}\!_n)}{h (\bom{X}\!_1 ,\ldots ,\bom{X}\!_n )}, \label{fdef} \\ 
 f_0 (\bom{X}\!_1, \ldots, \bom{X}\!_n) &= 
 \sum_{i=1}^{n-1}
 \biggr( \frac{\bom{X}\!_i}{\bom{X}\!_{i+1}} \biggr)^{q^{i-1}(q-1)}
 +\biggl( \frac{\bom{X}\!_n^{\,q^n}}{\bom{X}\!_1} \biggr)^{\frac{q-1}{q}} \label{f0def}. 
\end{align}
We simply write $f(\bom{X})$ for $f(\bom{X}\!_1, \ldots, \bom{X}\!_n)$, and 
$f(\bom{\xi}_r)$ for $f(\bom{\xi}_{r,1}, \ldots, \bom{\xi}_{r,n})$. 
We will use the similar notations also for other functions. 
We put 
\begin{equation}\label{Sdef}
 \bom{S} =f_0(\bom{X}) -f_0(\bom{\xi}^0_r ). 
\end{equation}

\begin{lem}\label{fZapp}
We have 
\[
 f(\bom{X}) \equiv f_0 (\bom{X}) \mod\!_>\, \frac{q-1}{nq} 
 \quad \textrm{and} \quad 
 \bom{S} \equiv f(\bom{X}) -f(\bom{\xi}^0_r ) \mod\!_>\, \frac{1}{n}. 
 \]
\end{lem}
\begin{proof}
We put 
\[
 f_1 (\bom{X}\!_1, \ldots, \bom{X}\!_n) =\sum_{i=1}^{n-3}
 \sum_{j=i+2}^{n-1}
 \biggl( \frac{\bom{X}\!_i}{\bom{X}\!_{i+1}} \biggr)^{q^{i-1}(q-1)}
 \biggl( \frac{\bom{X}\!_j}{\bom{X}\!_{j+1}} \biggr)^{q^{j-1}(q-1)}
 +\biggl( \frac{\bom{X}\!_n^{\,q^n}}{\bom{X}\!_1} \biggr)^{\frac{q-1}{q}} 
 \sum_{i=1}^{n-3}
 \biggl( \frac{\bom{X}\!_{i+1}}{\bom{X}\!_{i+2}} \biggr)^{q^i(q-1)}. 
\]
We note that 
$(\bom{X}\!_1,\ldots,\bom{X}\!_n)$ satisfies the assumption of 
Lemma \ref{mixed} by the definition of 
$\cB_r$ and Lemma \ref{exxi} (4). 
Then we see that 
\[
 f(\bom{X}) \equiv 
 1-\frac{\delta_0' (\bom{X})}{h (\bom{X})} \equiv 
 f_0 (\bom{X}) -f_1 (\bom{X}) \mod\!_>\, \frac{1}{n}
\]
using Lemma \ref{mixed} and the definition of $\delta_0'$. 
The claims follow from this, 
because 
\[
 v(f_1 (\bom{X})) \geq \frac{2(q-1)}{nq} \quad \textrm{and} \quad 
 v \bigl( f_1 (\bom{X}) -f_1 (\bom{\xi}^0_r ) \bigr) > \frac{2(q-1)}{nq} 
\]
hold. 
\end{proof}
We put 
$\bom{s}_i =(\bom{x}_i /\bom{x}_{i+1} )^{q^i (q-1)}$ 
for $1 \leq i \leq n-1$, and 
\begin{equation}\label{Yidef}
\begin{split}
 &\bom{s}_i \bom{s}_{n-1}^{-1} 
 =1 +\bom{Y}\!_i \quad 
 \textrm{for $1 \leq i \leq n-2$,} \quad 
 \bom{s}_{n-1}  
 =1 +\bom{Y}\!_{n-1}.  
\end{split}
\end{equation}
We put $m =\gcd (e,f)$ 
and 
\begin{equation}\label{zdef}
 \bom{z} = 
 \sum_{i=0}^{\frac{e}{m} -1} 
 \biggl( \frac{\bom{\theta}_r^{p^e} \bom{Y}\!_{n-1}}{\bom{\eta}_r} 
 \biggr)^{p^{im}} 
 -\frac{1}{n'} \sum_{i=0}^{\frac{f}{m} -1} 
 \biggl( \frac{\bom{S}}{\bom{\eta}_r} \biggr)^{p^{im}} .
\end{equation}

We put $f=m_0$ and $e=m_1$. 
We define $m_2, \ldots , m_{N+1}$ by the Euclidean algorithm as follows: 
We have 
\begin{align*}
 &m_{i-1} =n_i m_i +m_{i+1} \quad 
 \textrm{with} \ \ n_i \geq 0 \ \ \textrm{and} \ \ 
 0 \leq m_{i+1} < m_i 
 \quad \textrm{for} \ \ 1 \leq i \leq N, \\ 
 &m_N =m, \quad m_{N+1} =0. 
\end{align*}
We put 
\begin{equation}\label{T01}
 \bom{T}\!_0 = \frac{\bom{\theta}_r^{p^e} \bom{Y}\!_{n-1}}{\bom{\eta}_r}, 
 \quad 
 \bom{T}\!_1 = \frac{-\bom{S}}{n' \bom{\eta}_r}  
\end{equation}
and define $\bom{T}\!_2 ,\ldots , \bom{T}\!_N$ by 
\[
 \bom{T}\!_{i+1} = \bom{T}\!_{i-1} + \sum_{j=0}^{n_i -1} 
 \bom{T}\!_i^{\,p^{jm_i +m_{i+1}}} \quad 
 \textrm{for} \ 1 \leq i \leq N-1. 
\]
Then we see that 
\[
 \bom{z} = 
 \sum_{j=0}^{\frac{m_{i+1}}{m}-1} \bom{T}\!_i^{\,p^{jm}} + 
 \sum_{j=0}^{\frac{m_i}{m}-1} \bom{T}\!_{i+1}^{\,p^{jm}} 
 \quad \textrm{for} \ 1 \leq i \leq N-1 
\]
inductively by \eqref{zdef}. 
We see also that 
\begin{equation}\label{T0N}
 (-1)^{N -i} \bom{T}\!_i = 
 \sum_{j=0}^{\frac{m_i}{m} -1} \bom{T}\!_N^{\,p^{jm}} 
 + P_i (\bom{z}) 
\end{equation}
with some $P_i(x) \in \bZ [x]$ for $0 \leq i \leq N-1$. 
We put 
\begin{equation}\label{Ydef}
 \bom{Y}=\frac{(-1)^N \bom{\eta}_r}{\bom{\theta}_r^{p^e}} 
 \bom{T}\!_N^{\,p^{f-m}}. 
\end{equation}
Then we have 
\begin{equation}\label{YYn-1}
 \bom{Y} \equiv \bom{Y}\!_{n-1} \mod\!_>\, \frac{1}{n(p^e +1)} 
\end{equation}
by \eqref{T0N} and \eqref{Ydef}. 
We define a subaffinoid 
$\cB_r' \subset \cB_r$ 
by $v(\bom{z}) \geq 0$. 
We choose a square root 
$\bom{\eta}_r^{1/2} =(\eta_r^{q^{-j}/2})_{j \geq 0}$ 
and a $(p^e +1)$-st root 
$\bom{\eta}_r^{1/(p^e +1)} =(\eta_r^{q^{-j}/(p^e +1)})_{j \geq 0}$ 
of $\bom{\eta}_r$ compatibly. 
We set 
\begin{equation}\label{ydef}
\begin{split}
 &\bom{Y}\!_i = \bom{\eta}_r^{1/2} \bom{y}_i 
 \quad \textrm{with} \quad \bom{y}_i =(y_i^{q^{-j}})_{j \geq 0} 
 \quad \textrm{for $1 \leq i \leq n-2$}, \\ 
 &\bom{Y} =\bom{\eta}_r^{1/(p^e +1)} \bom{y} 
 \quad \textrm{with} \quad \bom{y} =(y^{q^{-j}})_{j \geq 0} 
\end{split}
\end{equation}
on $\cB_r'$. 
Let $\cB$ be the 
generic fiber of 
$\Spf \cO_{\bfC} \langle y^{1/q^{\infty}} ,y_1^{1/q^{\infty}}, \ldots ,
 y_{n-2}^{1/q^{\infty}}, z^{1/q^{\infty}} \rangle$. 
The parameters $\bom{y}, \bom{y}_1, \ldots ,\bom{y}_{n-2}, \bom{z}$ 
give the morphism 
$\Theta \colon \cB_r' \to \cB$. 
We simply say an analytic function on $\cB$ for 
a $q$-th power compatible system of analytic functions on 
$\cB$. 
We put 
\[
  1+ \bom{\theta}'_r = 
 (1+ \bom{\theta}_r)^{-n} (1 +n'\bom{\lambda}_r) 
 \biggl( \frac{\bom{\xi}^0_n}{\bom{\xi}'_n} \biggr)^{q^{n-1}}. 
\] 

\begin{lem}\label{BBisom}
The morphism $\Theta$ is an isomorphism. 
\end{lem}
\begin{proof}
We will construct the inverse morphism of $\Theta$. 
We can write $\bom{Y}\!_{n-1}$ and $\bom{S}$ as 
analytic functions on $\cB$ 
by \eqref{T01}, \eqref{T0N}, \eqref{Ydef} and \eqref{ydef}. 
Then we can write $\bom{x}_i /\bom{x}_{i+1}$ as 
an analytic function on $\cB$ by \eqref{Yidef}. 
By \eqref{f0def} and \eqref{Sdef}, we have 
\begin{align*}
 \frac{\bom{\eta}_r^{-(q-1)} \bom{S}^q}{(1+\bom{\theta}_r)^{(q-1)^2}} 
 =
 \sum_{i=1}^{n-2} (\bom{s}_i -1 ) 
 + \frac{\bom{s}_{n-1} -1}{(1+\bom{\theta}'_r )^{q-1}} 
 + (1+\bom{\theta}'_r )^{q(q-1)} 
 \Bigl( \bom{x}_n^{(q-1)(q^n -1)} \prod_{i=1}^{n-1} 
 ( \bom{x}_i^{-1} \bom{x}_{i+1} )-1 \Bigr). 
\end{align*}
By this equation, 
we can write $\bom{x}_n$ as 
an analytic functions on $\cB$. 
Hence, we have the inverse morphism of $\Theta$. 
\end{proof}
We put 
\[
 \delta_{\cB} (y,y_1, \ldots ,y_{n-2}, z) = 
 (\delta|_{\cB'_r} )\circ \Theta^{-1}
\]
equipped with its $q^j$-th root 
$\delta_{\cB}^{q^{-j}}$ for $j \geq 0$. 
We put 
\[
 \fX_r =\Spf \cO_{\bfC} \langle y^{1/q^{\infty}}, 
 y_1^{1/q^{\infty}}, \ldots ,
 y_{n-2}^{1/q^{\infty}}, z^{1/q^{\infty}} \rangle
 /(\delta_{\cB}^{q^{-j}} -t^{q^{-j}})_{j \geq 0}^-. 
\]
Let $\ol{\fX}_r$ denote the special fiber of 
$\fX_r$. 

\begin{thm}\label{thm:red}
The formal scheme $\fX_r$ is 
a formal model of $\cX_r$, 
and $\ol{\fX}_r$ is isomorphic to 
the perfection of the affine smooth variety defined by 
\begin{equation}\label{vareq}
 z^{p^m} -z =y^{p^e +1} -\frac{1}{n'} 
 \sum_{1 \leq i \leq j \leq n-2} y_i y_j \quad 
 \textrm{in} \ \bA_{k^{\rmac}}^n. 
\end{equation}
\end{thm}
\begin{proof}
Let $(\bom{X}\!_1,\ldots,\bom{X}\!_n)$ be the coordinate of $\cB_r$. 
By Lemma \ref{fZapp}, we have 
\begin{equation}\label{vfZ}
 v(f(\bom{X})) \geq \frac{q-1}{nq} \quad \textrm{and} 
 \quad 
 v(\bom{S} ) > \frac{q-1}{nq}. 
\end{equation}
We have 
\begin{equation}\label{hXX}
 h(\bom{X})^{q-1}= 
 \biggl( \frac{\bom{X}\!_n^{\,q^n}}{\bom{X}\!_1} \biggr)
 \prod_{i=1}^{n-1} \biggl( 
 \frac{\bom{X}\!_i}{\bom{X}\!_{i+1}} \biggr)^{q^i} . 
\end{equation}
We have 
\begin{equation}\label{XhhY}
 \biggl( \frac{\bom{X}\!_n^{\,q^n}}{\bom{X}\!_1} \biggr)^{q-1}
 = \Bigl( 
 \bom{\eta}_r(1 +\bom{\theta}_r )^{q-1} (1 +\bom{\theta}'_r )^q 
 \Bigr)^{q-1} 
 \biggl( \frac{h(\bom{X})}{h(\bom{\xi}^0_r)} 
 \biggr)^{(q-1)^2} 
 \prod_{i=1}^{n-1} \bom{s}_i^{-1}
\end{equation}
by \eqref{hXX}. 
We put 
\begin{equation}\label{Fdef}
 R(\bom{X})= 
 \frac{1-f(\bom{\xi}^0_r)}{1-f(\bom{X})} 
 - (1+\bom{S} ). 
\end{equation}
Then we have $v(R(\bom{X})) >1/n$ by 
Lemma \ref{fZapp} and \eqref{vfZ}. 
The equation 
$\delta (\bom{X}) =\delta (\bom{\xi}^0_r)$ is equivalent to 
\begin{equation}\label{XZFY}
 \biggl( \frac{\bom{X}\!_n^{\,q^n}}{\bom{X}\!_1} \biggr)^{q-1}
 = \Bigl( 
 \bom{\eta}_r (1 +\bom{\theta}_r )^{q-1} (1 +\bom{\theta}'_r )^q 
 \Bigr)^{q-1} 
 \bigl( 1 +\bom{S} +R(\bom{X} )  
 \bigr)^{(q-1)^2} 
 \prod_{i=1}^{n-1} \bom{s}_i^{-1}
\end{equation}
by \eqref{fdef}, \eqref{XhhY} and \eqref{Fdef}. 
We put 
\[
 F (\bom{X})= 
 (1 +\bom{\theta}'_r )^{q(q-1)} 
 \bigl( 1 + \bom{S} +R(\bom{X}) \bigr)^{(q-1)^2} 
 \prod_{i=1}^{n-1} \bom{s}_i^{-1} . 
\]
The equation 
\eqref{XZFY} is equivalent to 
\begin{equation}\label{f0q}
 f_0 (\bom{X})^q = 
 \bom{\eta}_r^{q-1} 
 (1+\bom{\theta}_r )^{(q-1)^2} 
 \Biggl( \sum_{i=1}^{n-2} \bom{s}_i 
 + \frac{\bom{s}_{n-1}}{(1+\bom{\theta}'_r )^{q-1}} 
 +F (\bom{X}) \Biggr) . 
\end{equation}
The equation 
\eqref{f0q} is equivalent to 
\begin{equation}\label{S0q}
 \bom{S}^q = 
 \bom{\eta}_r^{q-1} 
 (1+\bom{\theta}_r )^{(q-1)^2} 
 \Biggl( \sum_{i=1}^{n-2} (\bom{s}_i -1 )
 + \frac{\bom{s}_{n-1} -1}{(1+\bom{\theta}'_r )^{q-1}} 
 +F (\bom{X}) -F (\bom{\xi}^0_r ) \Biggr) . 
\end{equation}
We put 
\begin{align*}
 R_1 (\bom{X}) = {} 
 &(1+\bom{\theta}_r )^{(q-1)^2} 
 \Biggl( \sum_{i=1}^{n-2} (\bom{s}_i -1 )
 +\frac{\bom{s}_{n-1} -1}{(1+\bom{\theta}'_r )^{q-1}} 
 +F (\bom{X}) -F (\bom{\xi}^0_r ) \Biggr) \\ 
 &-\Biggl( \bom{S} 
 + \sum_{1 \leq i \leq j \leq n-2} \bom{Y}\!_i \bom{Y}\!_j 
 -n' \Bigl( \bom{Y}\!_{n-1}^{\,p^e +1} + 
 ( 1 +\bom{\theta}_r ) \bom{Y}\!_{n-1}^{\,p^e} 
 + \bom{\theta}_r^{p^e} \bom{Y}\!_{n-1} \Bigr) 
 \Biggr) . 
\end{align*}
Then we have $v(R_1 (\bom{X})) > 1/n$. 
The equation \eqref{S0q} is equivalent to 
\begin{equation}\label{Sq1a}
 \bom{S}^q = 
 \bom{\eta}_r^{q-1} \biggl( \bom{S} + 
 \sum_{1 \leq i \leq j \leq n-2} \bom{Y}\!_i \bom{Y}\!_j 
 -n' \Bigl( \bom{Y}\!_{n-1}^{\,p^e +1} + 
 ( 1 +\bom{\theta}_r ) \bom{Y}\!_{n-1}^{\,p^e} 
 + \bom{\theta}_r^{p^e} \bom{Y}\!_{n-1} \Bigr) 
 +R_1(\bom{X}) \biggr). 
\end{equation}
The equation \eqref{Sq1a} is equivalent to 
\begin{equation}\label{zyyn}
 \bom{z}^{p^m} -\bom{z} 
 = \bom{\eta}_r^{-1} \Biggl( \bom{Y}\!_{n-1}^{\,p^e +1} - 
 \frac{1}{n'} \sum_{1 \leq i \leq j \leq n-2} \bom{Y}\!_i \bom{Y}\!_j 
 - \frac{R_1(\bom{X})}{n'} \Biggr). 
\end{equation}

As a result, 
$\delta (\bom{X}) =\delta (\bom{\xi}_L)$ is equivalent to 
\eqref{zyyn} on $\cB_r$. 
By Lemma \ref{fZapp} and \eqref{zyyn}, 
we have $v(\bom{z}) \geq 0$ on $\cX_r$. 
This implies $\cX_r \subset \cB'_r$. 
We have the first claim by Lemma \ref{BBisom} and 
the construction of $\fX_r$. 
The second claim follows from \eqref{YYn-1} and \eqref{zyyn}. 
\end{proof}

\begin{rem}
If $n=p=2$, then the smooth compactification of 
the curve over $k$ defined by 
\eqref{vareq} is the supersingular elliptic curve, 
which appears as an irreducible component 
of a semi-stable reduction of 
a one-dimensional Lubin--Tate space in 
\cite{ITstab3} and \cite{ITreal3}. 
\end{rem}

\section{Group action on the reductions}\label{GroupRed}
\paragraph{Action of $\iGL_n$ and $D^{\times}$}
Let 
$\fI \subset M_n ( \cO_K )$ 
be the inverse image 
under the reduction map 
$M_n( \cO_K ) \to M_n(k)$ 
of the ring consisting of 
upper triangular matrices in $M_n (k)$. 

\begin{lem}\label{vgd}
Let $(g,d,1) \in G^0$. 
We take the integer $l$ such that 
$d \varphi_{D,r}^{-l} \in \cO_D^{\times}$. 
Let $(\bom{X}\!_1 ,\ldots ,\bom{X}\!_n )$ be the coordinate 
of $\cX_r$. 
Assume $v((g,d) \cdot \bom{X}\!_i) =v(\bom{X}\!_i)$ 
for $1 \leq i \leq n$ at some point of $\cX_r$. 
Then we have 
$(g,d ) \in (\varphi_{M,r},\varphi_{D,r})^l (\fI^{\times} \times \cO_D^{\times} )$. 
\end{lem}
\begin{proof}
This is proved in the same way as 
\cite[Lemma 3.1]{ITsimptame}. 
\end{proof}

We put  
\begin{equation}\label{g_r}
\mathbf{g}_r=
( \varphi_{M,r},\varphi_{D,r},1 )
 \in G. 
\end{equation}
We put 
\begin{equation}\label{eq:ep1}
 \varepsilon_1 = 
 \begin{cases}
 1 & \textrm{if $p^e = 2$,} \\ 
 0 & \textrm{if $p^e \neq 2$.} 
 \end{cases}
\end{equation}
For $a \in k^{\mathrm{ac}}$, 
we simply write $a$ also for the 
$q$-th power compatible system 
$(a^{q^{-j}})_{j \geq 0}$. 
\begin{prop}\label{frob}
\begin{enumerate}
\item\label{enu:actgr}
The action of $\bfg_r$ 
stabilizes $\cX_r$, and 
induces the automorphism of 
$\ol{\fX}_r$ defined by 
\begin{equation}\label{oct}
\begin{split}
 &(\bom{z},\bom{y},(\bom{y}_i)_{1 \leq i \leq n-2}) \\ 
 &\mapsto 
 \left( \bom{z} +\varepsilon_1 (\bom{y}_{n-2} +1 ), 
 \bom{y}, -\sum_{i=1}^{n-3} \bom{y}_i -2\bom{y}_{n-2} 
 +\varepsilon_1, 
 ( \bom{y}_{i-1} - \bom{y}_{n-2} +\varepsilon_1 
 )_{2 \leq i \leq n-2} \right) . 
\end{split}
\end{equation}
\item\label{enu:detgr}
Assume $p^e \neq 2$. 
Let $g_r \in 
\mathit{GL}_{n-1}(k)$ 
be the matrix corresponding to the action of 
$\bfg_r$ on $(\bom{y},(\bom{y}_i)_{1 \leq i \leq n-2})$ 
in \eqref{oct}. 
Then, $\det (g_r) =(-1)^{n-1}$. 
\end{enumerate}
\end{prop}
\begin{proof}
By \eqref{gl} and \eqref{divi}, 
we have 
\begin{equation}\label{grX} 
 \bfg_r^{\ast} \bom{X}_1 = 
 \bom{X}_n^{q^{n-1}}, \quad 
 \bfg_r^{\ast} \bom{X}_i =
 \bom{X}^{\frac{1}{q}}_{i-1} 
 \quad \textrm{for $2 \leq i \leq n$.}
\end{equation}
By \eqref{grX}, we have $\bfg_r^{\ast} (h(\bom{X})) =h(\bom{X})$. 
Hence, we have 
\begin{equation}\label{grS}
 \bfg_r^{\ast} \bom{S} \equiv \bom{S} \mod\!_{>}\, \frac{1}{n}
\end{equation}
by \eqref{fdef}, \eqref{Sdef} and Lemma \ref{fZapp}. 
By \eqref{XZFY} and \eqref{grX}, 
we have 
\begin{equation}\label{grs1}
 \bfg_r^{\ast} \bom{s}_1 \equiv 
 \prod_{i=1}^{n-1} \bom{s}_i^{-1} \mod\!_{>}\, \frac{1}{2n}. 
\end{equation}
We have also 
\begin{equation}\label{grsi}
\begin{split}
 \bfg_r^{\ast} \bom{s}_i = 
 \bom{s}_{i-1} \quad 
 \textrm{for $2 \leq i \leq n-2$,} \quad 
 \bfg_r^{\ast} \bom{s}_{n-1} = 
 \bom{s}_{n-2} (1+\bom{\theta}'_r )^{1-q} 
\end{split}
\end{equation}
by \eqref{grX}. 
We have 
\begin{equation}\label{grY1}
 \bfg_r^{\ast} \bom{Y}_1 \equiv 
 (1+\bom{\theta}_r )^n 
 (1+\bom{Y}_{n-2} )^{-2} 
 \prod_{i=1}^{n-3} (1+\bom{Y}_i )^{-1} 
 -1 
 \mod\!_{>}\, \frac{1}{2n} 
\end{equation}
by \eqref{Yidef}, \eqref{grs1} and \eqref{grsi}. 
We have also 
\begin{equation}\label{grYi}
\begin{split}
 \bfg_r^{\ast} \bom{Y}_i 
 &\equiv 
 (1+\bom{\theta}_r )^n 
 (1+\bom{Y}_{i-1} )
 (1+\bom{Y}_{n-2} )^{-1} -1 
 \mod\!_{>}\, \frac{1}{2n} \quad 
 \textrm{for $2 \leq i \leq n-2$,}\\ 
 \bfg_r^{\ast} \bom{Y}_{n-1} 
 &\equiv 
 (1+\bom{\theta}_r )^{-n} 
 (1+\bom{Y}_{n-2}) 
 (1+\bom{Y}_{n-1} ) -1 
 \mod\!_{>}\, \frac{1}{p^e n} 
\end{split}
\end{equation}
by \eqref{grsi}. 
The claim follows from \eqref{grS}, \eqref{grY1} and 
\eqref{grYi}. 
\end{proof}
 
Let $\mathfrak{P}$ be the Jacobson radical of 
the order $\mathfrak{I}$, and 
$\mathfrak{p}_D$ be the maximal ideal of 
$\mathcal{O}_D$. 
We put 
\[
 U_{\mathfrak{I}}^1=1+\mathfrak{P}, 
 \quad 
 U_D ^1 =1+\mathfrak{p}_D 
\] 
and
\[
 (U_{\mathfrak{I}}^1 \times U_D^1 )^1 = 
 \{ (g,d) \in U_{\mathfrak{I}}^1 \times U_D^1 \mid 
 \det (g)^{-1} \Nrd_{D/K} (d) =1 \}.  
\]
Let $\pr_{\cO_K/k} \colon \cO_K \to k$ 
be the reduction map. 
We put 
\[
 h_r (g,d) =\frac{1}{n'} 
 (\Tr_{k/\bF_{p^m}} \circ \pr_{\cO_K/k} ) 
 \Bigl( \Trd_{D/K}( \varphi_{D,r}^{-1}(d-1)) 
 -\tr (\varphi_{M,r}^{-1}(g-1)) 
 \Bigr) 
\]
for $(g,d) \in U_{\fI}^1 \times U_D^1$.

\begin{prop}\label{gda}
The stabilizer of $\cX_r$ in 
$\iGL_n (K)\times D^{\times}$ is 
$i_{\xi_r} (L_r^{\times}) \cdot (U_{\mathfrak{I}}^1 \times U_D^1 )^1$. 
Further, 
$(g,d) \in (U_{\mathfrak{I}}^1 \times U_D^1 )^1$ 
induces the automorphism of 
$\ol{\fX}_r$ defined by 
\[
 (\bom{z},\bom{y},(\bom{y}_i)_{1 \leq i \leq n-2}) \mapsto
 (\bom{z}+h_r (g,d) ,\bom{y}, (\bom{y}_i)_{1 \leq i \leq n-2}). 
\]
\end{prop}
\begin{proof}
Assume that 
$(g,d) \in \iGL_n (K)\times D^{\times}$ 
stabilizes $\mathcal{X}_r$. 
Then we have $\det (g)=\Nrd_{D/K} (d)$. 
We will show that 
\[
 (g,d) \in i_{\xi_r} (L_r^{\times}) 
 \cdot (U_{\mathfrak{I}}^1 \times U_D^1 )^1. 
\]
We have 
$(g,d ) \in (\varphi_{M,r},\varphi_{D,r})^l 
 (\fI^{\times} \times \cO_D^{\times} )$ 
for some integer $l$ by Lemma \ref{vgd}, 
since $(g,d)$ stabilizes $\mathcal{X}_r$ 
and we have 
$v(\bom{X}\!_i) =1/(n q^{i-1}(q-1))$ 
for $1 \leq i \leq n$ at any point of $\cX_r$ 
by Lemma \ref{exxi} (4) and \eqref{defcX}. 
Further, we may assume that 
$(g,d) \in \fI^{\times} \times \cO_D^{\times}$, 
since we already know that 
$(\varphi_{M,r},\varphi_{D,r})$ stabilizes $\mathcal{X}_r$ 
by Proposition \ref{frob} (1). 

We write 
$g=(a_{i,j})_{1 \leq i,j \leq n} \in \fI$ and 
$a_{i,j}=
\sum_{l=0}^{\infty}a_{i,j}^{(l)}\varpi_r^l$ with 
$a_{i,j}^{(l)} \in \mu_{q-1} (K) \cup\{0\}$.
By \eqref{gl}, we have 
\begin{equation}\label{x0}
\begin{split}
 & g^{\ast} \bom{X}_1 \equiv 
 a^{(0)}_{1,1} \bom{X}_1+a^{(1)}_{n,1} \bom{X}_n^{q^n} 
 \mod\!_>\, \frac{q}{n(q-1)},\\  
 & g^{\ast} \bom{X}_i \equiv 
 a^{(0)}_{i,i} \bom{X}_i+a_{i-1,i}^{(0)} \bom{X}_{i-1} 
 \mod\!_>\, \frac{1}{nq^{i-2}(q-1)} \quad 
 \textrm{for $2 \leq i \leq n$}.
\end{split}
\end{equation}
We write $d^{-1}=\sum_{i=0}^{\infty}d_i \varphi_{D,r}^i$ with 
$d_i \in \mu_{q^n-1}(K_n) \cup \{0\}$. We set $\kappa(d)=d_1/d_0$. 
By \eqref{divi}, we have 
\begin{equation}\label{xx2}
 d^\ast \bom{X}_i \equiv d_0 \bom{X}_i 
 \left(1+ \kappa(d) \bom{X}_i^{q-1} \right) 
 \mod\!_>\, \frac{1}{nq^{i-2}(q-1)} \quad 
 \textrm{for $1 \leq i \leq n$}.
\end{equation}
By \eqref{defcX}, \eqref{x0} and \eqref{xx2}, 
we have 
$(g,d) \in i_{\xi_r} (\cO_K^{\times}) \cdot (U_{\mathfrak{I}}^1 \times U_D^1 )^1$. 
Conversely, 
any element of 
$i_{\xi_r} (L_r^{\times}) \cdot (U_{\mathfrak{I}}^1 \times U_D^1 )^1$ 
stabilizes $\cX_r$ 
by Remark \ref{Ktri}, Proposition \ref{frob} 
and the above arguments. 

Let $(g,d) \in \cO_K^{\times} U_{\mathfrak{I}}^1 \times \cO_D^{\times}$. 
We put 
\begin{align*}
 \Delta_g (\bom{X}) 
 = 
 \sum_{i=1}^{n-1}\frac{a_{i,i+1}^{(0)}}{a_{i+1,i+1}^{(0)}}
 \biggl( \frac{\bom{X}_i}{\bom{X}_{i+1}} \biggr)^{q^i}
 +\frac{a_{n,1}^{(1)}\bom{X}_n^{q^n}}{a_{1,1}^{(0)}\bom{X}_1}, \quad 
 \Delta_d (\bom{X})
 =\sum_{i=1}^n\kappa(d)^{q^{i-1}} \bom{X}_i^{q^{i-1}(q-1)}.
\end{align*}
Then, we acquire 
\begin{equation}\label{gf0}
 f_0 \bigl( (g,d)^\ast \bom{X} \bigr) 
 \equiv f_0 (\bom{X}) 
 +\Delta_g (\bom{X}) +\Delta_d (\bom{X}) 
 \mod\!_{>}\, \frac{1}{n} . 
\end{equation}
We have 
\begin{equation}\label{kou}
 (g,d)^\ast \bom{S} \equiv \bom{S}+\Delta_g (\bom{X}) 
 +\Delta_d (\bom{X}) \mod\!_{>}\, \frac{1}{n} 
\end{equation}
by \eqref{Sdef} and \eqref{gf0}. 
We have 
\begin{equation}\label{gss}
 (g,d)^\ast \bom{s}_i \equiv \bom{s}_i \mod\!_{\geq}\, \frac{1}{n} 
\end{equation}
for $1 \leq i \leq n-2$. 
Let $(g,d) \in (U_{\mathfrak{I}}^1 \times U_D^1 )^1$. 
We obtain 
\[
 (g,d)^\ast \bom{z}=\bom{z}+h_r (g,d) 
\] 
by \eqref{zdef}, \eqref{kou} and \eqref{gss}. 
We can compute 
the action of $(g,d)$ on $\bom{y}$ and 
$\{\bom{y}_i\}_{1 \leq i \leq n-2}$ 
by \eqref{Yidef}, \eqref{YYn-1}, \eqref{ydef} and \eqref{gss}. 
\end{proof}

\paragraph{Action of the Weil group}
We put $\varphi_r' =\varphi_r^{p^e}$ and $E_r =K(\varphi_r')$. 
Let $\sigma \in W_{E_r}$ in this paragraph. 
We put 
\[
 a_{\sigma}=\mathrm{Art}_{E_r}^{-1} (\sigma) \quad \textrm{and} \quad 
 u_{\sigma}=a_{\sigma} {\varphi_r'}^{-n_{\sigma}} 
 \in \mathcal{O}_{E_r}^{\times}. 
\]
We take 
\[
 b_{\sigma} \in \mu_{q-1}(K) \quad \textrm{such that} \quad 
 \bar{b}_{\sigma}^{p^e} = \bar{u}_{\sigma} \in k. 
\]
We put 
\[
 c_{\sigma} =b_{\sigma}^{-n} \Nr_{E_r/K} (u_{\sigma} ) \in U_K^1. 
\]
Let 
\[
 g_{\sigma}=(a_{i,j})_{1 \leq i,j \leq n} 
 \in \mathcal{O}_K^{\times} U_{\mathfrak{I}}^1 
\]
be the element defined by
$a_{i,i}=b_{\sigma}$ for $1 \leq i \leq n-1$, 
$a_{n,n}=b_{\sigma} c_{\sigma}$ and 
$a_{i,j}=0$ if $i \neq j$. 
We put 
\begin{equation}\label{gsig}
 \bfg_{\sigma} = 
 (g_{\sigma},\varphi_{D,r}^{-n_{\sigma}},\sigma) \in G. 
\end{equation} 
Then $\bfg_{\sigma}$ stabilizes each component in 
\eqref{Mdec}. 
We choose elements 
$\alpha_r, \beta_r, \gamma_r \in K^{\mathrm{ac}}$ such that 
\begin{equation}\label{extgen}
\begin{split}
 &\alpha_r^{p^e +1} = - \varphi_r' , \quad 
 \beta_r^{p^{2e}} +\beta_r 
 = - \alpha_r^{-1} , \quad
 \gamma_r^{p^m} - \gamma_r =\beta_r^{p^e+1} +\varepsilon_0 ,\\ 
 &\alpha_r^{-1} \eta_r^{\frac{p^e}{p^e +1}} \equiv 1 , \quad 
 \beta_r^{-1} \theta_r^{p^e} \eta_r^{-\frac{p^e}{p^e +1}} 
 \equiv 1 ,\quad 
 \gamma_r^{-1} \sum_{i=0}^{\frac{f}{m} -1} 
 (\lambda_r \eta_r^{-1})^{p^{im}} \equiv 1 
 \mod\!_{>}\, 0. 
\end{split}
\end{equation}
For $\sigma \in W_{E_r}$, 
we set 
\begin{gather*}\label{abcdef}
\begin{aligned}
 a_{r,\sigma} =\frac{\sigma(\alpha_r)}{\alpha_r} , \quad
 b_{r,\sigma} =a_{r,\sigma} \sigma(\beta_r)-\beta_r, \quad 
 c_{r,\sigma}  =\sigma(\gamma_r)-\gamma_r + 
 \sum_{i=0}^{\frac{e}{m}-1}
 ( b_{r,\sigma}^{p^e} (\beta_r + b_{r,\sigma} ) )^{p^{im}}. 
\end{aligned}
\end{gather*}
Then we have 
$a_{r,\sigma}, b_{r,\sigma}, c_{r,\sigma} \in \cO_{\bfC}$ 
and 
\begin{equation}\label{eq:rel}
\begin{split}
 &a_{r,\sigma} \equiv 
 \left( 
 \frac{\sigma(\eta_r^{\frac{1}{p^e +1}})}{\eta_r^{\frac{1}{p^e +1}}} 
 \right)^{p^e}, \quad 
 b_{r,\sigma} \equiv \left( 
 \frac{\sigma(\theta_r)-\theta_r}{\eta_r^{\frac{1}{p^e +1}}}
 \right)^{p^e} ,\\ 
 &c_{r,\sigma} \equiv 
 \sum_{i=0}^{\frac{f}{m}-1} \left( 
 \frac{\sigma(\lambda_r)-\lambda_r}{\eta_r} \right)^{p^{im}} 
 -\sum_{i=0}^{\frac{e}{m}-1} \left( \frac{(\sigma (\theta_r) -\theta_r) 
 \sigma (\theta_r )^{p^e}}{\eta_r} \right)^{p^{im}} 
 \mod\!_{>}\, 0 
\end{split}
\end{equation}
by \eqref{defthelam} and \eqref{extgen}. 
Let
\[
 Q= \Bigl\{
 g(a,b,c) \ \Big| \ 
 a, b, c \in k^{\mathrm{ac}}, \ 
 a^{p^e +1} =1, \ 
 b^{p^{2e}}+b=0, \  
 c^{p^m} - c + b^{p^e+1} =0
 \Bigr\}
 \]
be the group whose multiplication is given by 
\[
 g(a_1,b_1,c_1) \cdot g(a_2,b_2,c_2)=
 g \biggl( a_1 a_2 ,a_1 b_2 +b_1 , 
 c_1 + c_2 + \sum_{i=0}^{\frac{e}{m} -1}
 (a_1 b_1^{p^e} b_2 )^{p^{im}} \biggr) . 
\]
Let $Q \rtimes \mathbb{Z}$ be the semidirect product, where  
$l \in \mathbb{Z}$ 
acts  on $Q$ by 
$g(a,b,c) 
\mapsto g(a^{q^{-l}},b^{q^{-l}},c^{q^{-l}})$. 
Let 
$(g(a,b,c),l) \in Q \rtimes \mathbb{Z}$ 
act on $\ol{\fX}_r$ by 
\begin{equation}\label{QZactXr}
 (\bom{z},\bom{y},(\bom{y}_i)_{1 \leq i \leq n-2}) 
 \mapsto 
 \biggl( \biggl( \bom{z} +\sum_{i=0}^{\frac{e}{m} -1} 
 (b \bom{y})^{p^{im}} +c \biggr)^{q^l},
 (a (\bom{y}+b^{p^e}))^{q^l}, 
 ( 
 a^{\frac{p^e +1}{2}} 
 \bom{y}_i^{q^l} )_{1 \leq i \leq n-2} \biggr) . 
\end{equation} 
We have the surjective homomorphism 
\begin{equation}\label{hom}
\Theta_r \colon W_{E_r} \to Q \rtimes \mathbb{Z};\ 
\sigma \mapsto \left(g(\bar{a}_{r,\sigma},\bar{b}_{r,\sigma},\bar{c}_{r,\sigma}),n_{\sigma}\right).
\end{equation}

\begin{prop}\label{LT}
Let $\sigma \in W_{E_r}$. 
Then, 
$\bfg_{\sigma} \in G$ 
stabilizes $\cX_r$, and 
induces the automorphism of 
$\ol{\fX}_r$ given by 
$\Theta_r (\sigma)$. 
\end{prop}
\begin{proof}
Let $P \in \mathcal{X}_r(\mathbf{C})$. 
We have 
\begin{align}
 \bom{S} (P \bfg_{\sigma} ) &= 
 f_0 \bigl( \bom{X}(P
 \bfg_{\sigma}) \bigr) - 
 f_0 (\bom{\xi}^0_r ) \notag \\
 &= f_0 \bigl( \bom{X}(P
 \bfg_{\sigma}) \bigr) - 
 f_0 \bigl( \bom{X}(P(1,\varphi_{D,r}^{-n_{\sigma}},\sigma)) \bigr) + 
 \sigma^{-1} \bigl( f_0 ( \bom{X}(P) ) \bigr) - 
 f_0 (\bom{\xi}^0_r ) \notag \\ 
 &\equiv \Delta_{g_{\sigma}} 
 \bigl( \bom{X}(P(1,\varphi_{D,r}^{-n_{\sigma}},\sigma)) \bigr) + 
 \sigma^{-1} \bigl( \bom{S}(P) +f_0 (\bom{\xi}^0_r ) \bigr) - 
 f_0 (\bom{\xi}^0_r ) \notag \\ 
 &\equiv \sigma^{-1} \bigl( \bom{S}(P) \bigr) 
 +f_0 (\sigma^{-1} (\bom{\xi}^0_r ) ) - 
 f_0 (\bom{\xi}^0_r ) \mod\!_{>}\, \frac{1}{n} \label{ZPsig}
\end{align}
by \eqref{Sdef} and \eqref{gf0}. 
We have 
\begin{equation}\label{Delxi}
 f_0 (\sigma^{-1} (\bom{\xi}^0_r )) - f_0 (\bom{\xi}^0_r ) 
 \equiv n'(\sigma^{-1}(\bom{\lambda}_r) -\bom{\lambda}_r ) 
 \mod\!_{>}\, \frac{1}{n} . 
\end{equation}
We put $s_i (\bom{X})=(\bom{X}_i/\bom{X}_{i+1})^{q^i(q-1)}$ 
for $1 \leq i \leq n-1$. 
We have 
\begin{align}
 &s_{n-1} (\bom{\xi}^0_r ) \bom{Y}_{n-1} (P \bfg_{\sigma} ) = 
 s_{n-1} \bigl( \bom{X}(P
 \bfg_{\sigma}) \bigr) - 
 s_{n-1} (\bom{\xi}^0_r ) \notag \\
 &= s_{n-1} \bigl( \bom{X}(P
 \bfg_{\sigma}) \bigr) - 
 s_{n-1} \bigl( \bom{X}(P(1,\varphi_{D,r}^{-n_{\sigma}},\sigma)) \bigr) + 
 \sigma^{-1} \bigl( s_{n-1} ( \bom{X}(P) ) \bigr) - 
 s_{n-1} (\bom{\xi}^0_r ) \notag \\ 
 &\equiv 
 \sigma^{-1} \bigl( s_{n-1} (\bom{\xi}^0_r ) \bom{Y}_{n-1} (P) \bigr) + 
 \sigma^{-1} \bigl( s_{n-1} (\bom{\xi}^0_r )  \bigr) - 
 s_{n-1} (\bom{\xi}^0_r ) 
 \mod\!_{>}\, \frac{q-1}{n} + \frac{1}{np^e} \label{Yn-1sigg}
\end{align}
by \eqref{Yidef} and \eqref{gf0}. 
Hence, we have 
\begin{equation}\label{Yn-1theta}
 \bom{Y}_{n-1} (P \bfg_{\sigma} ) \equiv 
 \sigma^{-1} \bigl( \bom{Y}_{n-1} (P) \bigr) + 
 \sigma^{-1} (\bom{\theta}_r ) - \bom{\theta}_r 
 \mod\!_{>}\, \frac{1}{np^e} . 
\end{equation}
We put 
$\bom{\theta}_{r,\sigma} = \sigma (\bom{\theta}_r ) - \bom{\theta}_r$ 
and 
$\bom{\lambda}_{r,\sigma} = \sigma (\bom{\lambda}_r ) - \bom{\lambda}_r$. 
We have 
\begin{align*}
 \sigma \bigl( \bom{z}(P \bfg_{\sigma} ) \bigr) 
 &=
 \sigma \Biggl( \sum_{i=0}^{\frac{e}{m} -1} 
 \biggl( \frac{\bom{\theta}_r^{p^e} 
 \bom{Y}\!_{n-1}(P \bfg_{\sigma} )}{\bom{\eta}_r} 
 \biggr)^{p^{im}} -\frac{1}{n'} 
 \sum_{i=0}^{\frac{f}{m} -1} 
 \biggl( \frac{\bom{S}(P \bfg_{\sigma} )}{\bom{\eta}_r} \biggr)^{p^{im}} \Biggr) 
 \\ 
 &\equiv \bom{z}(P) + 
 \sum_{i=0}^{\frac{e}{m} -1} 
 \biggl( \frac{\bom{\theta}_{r,\sigma}^{p^e} 
 \bom{Y}\!_{n-1} (P) -\sigma(\bom{\theta}_r^{p^e}) 
 \bom{\theta}_{r,\sigma}}{\bom{\eta}_r} 
 \biggr)^{p^{im}} 
 +\sum_{i=0}^{\frac{f}{m} -1} 
 \biggl( \frac{\bom{\lambda}_{r,\sigma}}{\bom{\eta}_r} \biggr)^{p^{im}} \\ 
 &\equiv \bom{z}(P) + 
 \sum_{i=0}^{\frac{e}{m} -1} 
 (b_{r,\sigma} \bom{y} (P) )^{p^{im}} 
 +c_{r,\sigma} 
 \mod\!_{>}\, 0 
\end{align*}
by \eqref{zdef}, \eqref{eq:rel}, \eqref{ZPsig}, 
\eqref{Delxi}, \eqref{Yn-1sigg} and 
\eqref{Yn-1theta}. 
We see also that 
\begin{align*}
 \sigma \biggl( 
 \frac{\bom{Y}_{n-1} (P \bfg_{\sigma} )}{\bom{\eta}_r^{1/(p^e +1)}} \biggr) 
 \equiv a_{r,\sigma}^{-p^e} (\bom{y}-b_{r,\sigma}^{\frac{1}{p^e}})
 \equiv a_{r,\sigma} (\bom{y}+b_{r,\sigma}^{p^e})  \mod\!_{>}\, 0 
\end{align*}
by \eqref{eq:rel} and \eqref{Yn-1theta}. 
By the same argument using \eqref{xx2}, we have 
\[
 \bom{Y}_i(P \bfg_{\sigma} ) 
 \equiv \sigma^{-1}(\bom{Y}_i(P)) \mod\!_{>}\, \frac{1}{2n} 
\]
for $1 \leq i \leq n-1$. 
This implies 
\begin{equation*}
 \bom{y}_i(P \bfg_{\sigma} ) \equiv 
 \frac{\sigma^{-1}( \bom{\eta}_r^{1/2})}{\bom{\eta}_r^{1/2}} 
 \sigma^{-1}(\bom{y}_i(P))
 \equiv 
 a_{r,\sigma}^{(p^e +1)/2} \bom{y}_i(P)^{q^{n_{\sigma}}} \mod\!_{>}\, 0
\end{equation*}
for $1 \leq i \leq n-1$ by \eqref{eq:rel}. 
\end{proof}

\paragraph{Stabilizer}
We put 
$n_1 =\gcd (n,p^m-1)$. 
We put 
\[
 \varphi_r'' =\varphi_r'^{n_1} \quad \textrm{and} \quad 
 F_r =K(\varphi_r''). 
\]
Let $\sigma \in W_{F_r}$. 
We put 
\[
 \zeta_{\sigma} =\frac{\sigma^{-1}(\varphi_r')}{\varphi_r'}. 
\]
Let $\zeta_{\sigma}^{1/p^e}$ be the 
$p^e$-th root of $\zeta_{\sigma}$ in $\mu_{p^m -1}(K)$. 
We put 
\[
 \varphi_{r,\sigma} =\zeta_{\sigma}^{1/p^e} \varphi_r. 
\]
Let $\cG_{r,\sigma}$ 
be the one-dimensional formal $\cO_{L_r}$-module 
over $\cO_{\wh{L}_r^{\rm ur}}$ 
defined similarly to $\cG_r$ 
changing $\varphi_r$ by $\varphi_{r,\sigma}$. 
We take a compatible system 
$\{t_{r,j,\sigma}\}_{j \geq 1}$ in $\bfC$ 
such that 
\[
 \frac{\sigma^{-1} (t_{r,1})}{t_{r,1,\sigma}} \equiv 1 \mod\!_{>}\, 0, \quad 
 [\varphi_{r,\sigma}]_{\cG_{r,\sigma}}(t_{r,1,\sigma})=0, \quad 
 [\varphi_{r,\sigma}]_{\cG_{r,\sigma}}(t_{r,j,\sigma})=t_{r,j-1,\sigma} 
\]
for $j \geq 2$. 
We construct 
$\xi_{r,\sigma}$ as in Lemma \ref{exxi} 
using $\{t_{r,j,\sigma}\}_{j \geq 1}$. 
Then $\xi_{r,\sigma}$ has CM by $L_r$.

\begin{lem}\label{aprbyCM}
For $\sigma \in W_{F_r}$, 
we have 
\begin{align*}
 \frac{\sigma^{-1} ( \xi_{r,i})}{\xi_{r,\sigma,i}} 
 &\equiv 1 
 \mod\!_{\geq}\, \frac{1}{q^{i-1}p^{e-1}(p-1)} \quad 
 \textrm{for $1 \leq i \leq n$,} \\ 
 \sigma^{-1} (\bom{\theta}_r) &\equiv 
 \bom{\theta}_r 
 \mod\!_{\geq}\, \frac{1}{n(p^e+1)} . 
\end{align*}
\end{lem}
\begin{proof}
We have 
\begin{equation}\label{sigphi}
 \frac{\sigma^{-1} (\varphi_r)}{\varphi_r} 
 \equiv \zeta_{\sigma}^{1/p^e} 
 \mod\!_{\geq}\, \frac{1}{p^{e-1}(p-1)}. 
\end{equation}
We obtain the claims by 
\eqref{sigphi} and 
\[
 \bigl( 
 \sigma^{-1} (\bom{\theta}_r) - \bom{\theta}_r 
 \bigr)^{p^{2e}} 
 +\bom{\eta}_r^{p^e -1} \bigl( 
 \sigma^{-1} (\bom{\theta}_r) - \bom{\theta}_r 
 \bigr) + 
 \bigl( 
 \bom{1}+\sigma^{-1} (\bom{\theta}_r) \bigr) 
 \bigl( 
 \sigma^{-1} (\bom{\eta}_r)^{p^e -1} 
 -\bom{\eta}_r^{p^e -1} \bigr) =\bom{0}, 
\]
which follows from \eqref{defthelam}. 
\end{proof}

We define 
$j_r \colon W_{F_r} \to L_r^{\times} \backslash 
 (\iGL_n (K) \times D^{\times} )$ as follows: 
\begin{quote}
Let $\sigma \in W_{F_r}$. 
Since $\xi_{r,\sigma}$ has CM by $L_r$, 
there exists $(g,d) \in \iGL_n (K) \times D^{\times}$ 
uniquely up to left multiplication by $L_r^{\times}$ 
such that 
$(g,d,1) \in G^0$ and 
$\xi_{r,\sigma} (g,d,1) =\xi_r$ 
by Lemma \ref{CMtrans}. 
We put 
$j_r (\sigma )=L_r^{\times} (g,\varphi_{D,r}^{-n_{\sigma}} d)$. 
\end{quote}
For $\sigma \in W_{L_r}$, we put 
$a_{\sigma}=\mathrm{Art}_{L_r}^{-1}(\sigma) \in L_r^{\times}$ 
and 
$u_{\sigma}=a_{\sigma} \varphi_{r}^{-n_{\sigma}} 
 \in \mathcal{O}_{L_r}^{\times}$. 
\begin{lem}\label{jLsig}
For $\sigma \in W_{L_r}$, we have 
$j_r (\sigma)=L_r^{\times} (1,a_{\sigma}^{-1})$. 
\end{lem}
\begin{proof}
This follows from \cite[Lemma 3.1.3]{BWMax}. Note that 
our action of $W_K$ is inverse to that in \cite{BWMax}. 
\end{proof}

We put 
\[
 \cS_r = \{ (g,d,\sigma) \in G \mid \sigma \in W_{F_r}, \ 
 j_r (\sigma )=L_r^{\times} (g,d) \}. 
\]

\begin{lem}\label{SLact}
The action of $\cS_r$ on $\cM^{(0)}_{\infty,\ol{\eta}}$ stabilizes $\cX_r$, 
and induces the action on $\ol{\fX}_r$. 
\end{lem}
\begin{proof}
We take an element of $\cS_r$, 
and write it as $(g,\varphi_{D,r}^{-n_{\sigma}} d,\sigma )$, 
where $(g,d,1) \in G^0$ and $\sigma \in W_{F_r}$. 
Since 
$\xi_{r,\sigma} (g,d,1) =\xi_r$, 
we have 
$(g,d ) \in (\varphi_{M,r},\varphi_{D,r})^l (\fI^{\times} \times \cO_D^{\times} )$ 
by Lemma \ref{vgd} and Lemma \ref{aprbyCM}. 

To show the claims, 
we may assume that 
$(g,d ) \in \fI^{\times} \times \cO_D^{\times}$ 
by Proposition \ref{frob} \ref{enu:actgr}. 
We write 
$g=(a_{i,j})_{1 \leq i,j \leq n} \in \fI^{\times}$ 
and 
$a_{i,j}=
\sum_{l=0}^{\infty}a_{i,j}^{(l)}\varpi_r^l$ with 
$a_{i,j}^{(l)} \in \mu_{q-1} (K) \cup\{0\}$, 
and 
$d^{-1}=\sum_{i=0}^{\infty}d_i \varphi_{D,r}^i$ with 
$d_i \in \mu_{q^n-1}(K_n) \cup \{0\}$. 
For $1 \leq i \leq n-1$, 
we have 
\begin{equation}\label{aid0}
 \frac{a_{i,i}^{(0)}}{a_{i+1,i+1}^{(0)}} =d_0^{q-1}     
\end{equation}
by 
$\xi_{r,\sigma} (g,d,1) =\xi_r$ 
using \eqref{x0}, \eqref{xx2},  
$\xi_{r,\sigma,i}=\xi_{r,\sigma,i+1}^q$ and 
$\xi_{r,i}=\xi_{r,i+1}^q$. 
The condition on the first line in 
\eqref{defcX} is equivalent to 
\begin{equation}\label{cXdefeq1}
 v \biggl( \frac{\bom{X}_i}{\bom{X}_{i+1}} -
 \Bigl( \frac{\bom{X}_{n-1}}{\bom{X}_n} \Bigr)^{q^{n-1-i}} 
 \biggr) 
 \geq \frac{3}{2nq^i} \quad 
 \textrm{for $1 \leq i \leq n-2$.}     
\end{equation}
We see that 
the condition \eqref{cXdefeq1} 
is stable under the action of 
$(g,\varphi_{D,r}^{-n_{\sigma}} d,\sigma )$ 
using \eqref{x0} and \eqref{xx2}, 
because 
$a_{i,i}^{(0)}/a_{i+1,i+1}^{(0)}$ 
is independent of $i$ by \eqref{aid0}. 
We see that 
the condition on the second line in 
\eqref{defcX} 
is stable 
under the action of 
$(g,\varphi_{D,r}^{-n_{\sigma}} d,\sigma )$ 
by Lemma \ref{aprbyCM} 
using \eqref{x0} and \eqref{xx2}. 
\end{proof}

The group $\cS_r$ normalizes 
$i_{\xi_r} (L_r^{\times}) \cdot (U_{\mathfrak{I}}^1 \times U_D^1 )^1$ 
by Proposition \ref{gda}. 
We put 
\[
 H_r =(U_{\mathfrak{I}}^1 \times U_D^1 )^1 
 \cdot \cS_r \subset G. 
\]
Then $H_r$ acts on $\ol{\fX}_r$ by Lemma \ref{SLact} 
and the proof of Proposition \ref{gda}. 

\begin{prop}
The subgroup $H_r \subset G^0$ is the stabilizer of 
$\cX_r$ in $\cM^{(0)}_{\infty,\ol{\eta}}$. 
\end{prop}
\begin{proof}
Assume that 
$(g,\varphi_{D,r}^{-n_{\sigma}} d,\sigma ) \in G^0$  stabilizes $\cX_r$. 
It suffices to show that 
\[
 (g,\varphi_{D,r}^{-n_{\sigma}} d,\sigma ) \in H_r. 
\]
By Lemma \ref{vgd}, 
we have 
$(g,d ) \in (\varphi_{M,r},\varphi_{D,r})^l (\fI^{\times} \times \cO_D^{\times} )$. 
Hence, 
we may assume that 
$(g,d ) \in \fI^{\times} \times \cO_D^{\times}$ 
by Proposition \ref{frob} \ref{enu:actgr}. 

First, we show that $\sigma \in W_{F_r}$. 
We write 
$g=(a_{i,j})_{1 \leq i,j \leq n} \in \fI^{\times}$, 
$a_{i,j}=
\sum_{l=0}^{\infty}a_{i,j}^{(l)}\varpi_r^l$ 
and 
$d^{-1}=\sum_{i=0}^{\infty}d_i \varphi_{D,r}^i$ 
as in the proof of Lemma \ref{SLact}.  
Since 
$(g,\varphi_{D,r}^{-n_{\sigma}} d,\sigma )$ 
stabilizes $\cX_r$, 
we have 
\begin{align}
 \frac{a_{i,i}^{(0)}}{a_{i+1,i+1}^{(0)}} =d_0^{q-1} \quad 
 \textrm{for $1 \leq i \leq n-1$}, \label{aid02} \\ 
 \frac{a_{n,n}^{(0)} d_0 \sigma^{-1}(\xi_{r,n}^0)}{\xi_{r,n}^0} 
 \equiv 1 
 \mod\!_{\geq}\, \frac{1}{nq^{n-1}(p^e +1)} \label{adxin} 
\end{align}
by \eqref{defcX}, \eqref{x0}, \eqref{xx2} and 
$\xi_{r,i}=\xi_{r,i+1}^q$. 
By taking the $p^e q^{n-1}(q-1)$-st power of \eqref{adxin}, 
we see that 
\begin{equation}\label{dsigphi'}
 d_0^{p^e q^{n-1}(q-1)} 
 \frac{\sigma^{-1}(\varphi_{r}')}{\varphi_{r}'} 
 \equiv \biggl( 
 \frac{1+\theta_r}{1+\sigma^{-1} (\theta_r)} \biggr)^{p^e(q-1)} 
 \mod\!_{\geq}\, \frac{p^e}{n(p^e +1)}.     
\end{equation}
This implies that 
the left hand side of 
\eqref{dsigphi'} is equal to $1$. 
Hence we have 
$\sigma^{-1} (\varphi_r')/\varphi_r' \in \mu_{q-1}(K)$ and 
$\sigma^{-1} (\theta_r) \equiv \theta_r \mod\!_{\geq}\, 1/(n(p^e +1))$, 
since $d_0^{q-1} \in \mu_{q-1}(K)$ by \eqref{aid02}. 
These happen only if 
$\sigma \in W_{F_r}$ by the proof of Lemma \ref{aprbyCM} 
and $\mu_{p^e-1}(K^{\rm ur}) \cap \mu_{q-1}(K)=\mu_{p^m-1}(K)$. 
Since $\sigma \in W_{F_r}$, 
we may assume that $\sigma =1$ 
by Lemma \ref{SLact}. 
Then $(g,d,1) \in H_r$ by 
Proposition \ref{gda}. 
\end{proof}

\section{Artin--Schreier variety}\label{CompAS}
\subsection{Tate conjecture}\label{ssec:Tate}
Let $m$ be a positive integer such that $\bF_{p^m} \subset \bF_q$. 
Let $N$ be a positive even integer. 
We put $n_0 =N/2$. 
We consider the affine smooth variety 
$X_{N,\bF_q}$ over $\bF_q$ defined by 
\[
 z^{p^m}-z=\sum_{i=1}^{n_0} u_{2i-1} u_{2i} \quad \textrm{in} \  \bA^{N+1}_{\bF_q} . 
\]
Let $X_{N}$ be the base change of $X_{N,\bF_q}$ to 
$\overline{\mathbb{F}}_q$. 
For an integer $i \geq 0$, 
we simply write $\mathbb{A}^i$ for the affine space 
$\mathbb{A}_{\ol{\bF}_q}^i$. 

\begin{rem}\label{rem:XQN}
Let $Q(y_1,\ldots,y_N)$ 
be any non-degenerate quadratic form on $\mathbb{A}^N$. 
Then the affine smooth 
variety over $\overline{\mathbb{F}}_q$ 
defined by 
\[
 z^{p^m}-z=Q(y_1,\ldots,y_N) \quad \textrm{in} \  \bA^{N+1} 
\]
is isomorphic to $X_N$ 
by \cite[XII, Proposition 1.2]{SGA7-2}. 
\end{rem}

For each $\zeta \in \mathbb{F}_{p^m}^{\times}$, 
we consider the homomorphism 
\[
 p_{\zeta} \colon \mathbb{F}_{p^m}
 \to \mathbb{F}_p;\ 
 x \mapsto 
 \Tr_{\mathbb{F}_{p^m}/\mathbb{F}_p}(\zeta^{-1} x) . 
\]
Then, we consider the quotient 
$X_{N,\zeta}=X_N/\ker p_{\zeta}$. 
Note that 
the quotient $X_{N,\zeta}$ depends only on 
the class $[\zeta] \in \bF_{p^m}^{\times}/\bF_p^{\times}$ of $\zeta$. 
The variety $X_{N,\zeta}$ has the defining equation 
\begin{equation}\label{fundd}
 \zeta ( z_{\zeta}^p -z_{\zeta} )=\sum_{i=1}^{n_0} u_{2i-1} u_{2i} 
 \quad \textrm{in $\mathbb{A}^{N+1}$}, 
\end{equation}
where the relation between $z$ and $z_{\zeta}$ 
is given by $z_{\zeta} =\sum_{i=0}^{m-1}(\zeta^{-1} z)^{p^i}$. 
Let $\ell \neq p$ be a prime number. 
For a topological abelian group $A$, 
let $A^{\vee}$ denote the set of the smooth characters 
$A \to \ol{\bQ}_{\ell}^{\times}$. 
Let $\cL_{\psi}$ be the Artin--Schreier 
$\ol{\bQ}_{\ell}$-sheaf on 
$\bA^1$ associated to $\psi \in \bF_{p^m}^{\vee}$, 
which is $\fF (\psi)$ in the notation of 
\cite[Sommes trig.\ 1.8 (i)]{DelCoet}. 
For a polynomial $f \in \ol{\bF}_q [x_1, \ldots, x_l]$, 
let $\cL_{\psi}(f)$ denote the pullback of $\cL_{\psi}$ 
under $f \colon \bA^l \to \bA^1$. 

\begin{lem}\label{AS0}
We have an isomorphism 
\[
 \bigoplus_{[\zeta] \in \mathbb{F}_{p^m}^{\times}/\bF_p^{\times}}
 H^{N}(X_{N,\zeta},\overline{\mathbb{Q}}_{\ell})
 \simeq 
 H^N(X_N,\overline{\mathbb{Q}}_{\ell})
\]
induced by the pullbacks 
and 
$\dim H^{N}(X_{N,\zeta},\overline{\mathbb
{Q}}_{\ell})=p-1$. 
\end{lem}
\begin{proof}
For $\psi \in \bF_{p^m}^{\vee} \setminus \{ 1 \}$, we have 
\[
 H_{\mathrm{c}}^{i} (\bA^2, \cL_{\psi}(xy)) = 
 \begin{cases}
 \overline{\mathbb{Q}}_{\ell}(-1) & \textrm{if $i=2$,}\\ 
 0 & \textrm{otherwise} 
 \end{cases}
\]
by \cite[Proposition 1.2.2.2]{LauTFcW} as in the proof of 
\cite[Lemma 2.1]{ITGeomHW}. 
Hence, by the K\"{u}nneth formula, 
we have isomorphisms 
\[
 H_{\mathrm{c}}^{N}(X_N,\overline{\mathbb{Q}}_{\ell}) 
 \simeq 
 \bigoplus_{\psi \in \mathbb{F}_{p^m}^{\vee}
 \setminus\{1\}} 
 H_{\mathrm{c}}^{N} \left( \bA^N,\cL_{\psi}\left( \sum_{i=1}^{n_0} u_{2i-1} u_{2i} \right) \right) 
 \simeq 
 \bigoplus_{\psi \in \mathbb{F}_{p^m}^{\vee}
 \setminus\{1\}} \psi 
\]
as $\mathbb{F}_{p^m}$-representations. 
By Poincar\'{e} duality, we have an isomorphism 
\[
 H^{N}(X_N,\overline{\mathbb{Q}}_{\ell})
 \simeq \bigoplus_{\psi \in \mathbb{F}_{p^m}^{\vee}
 \setminus\{1\}} \psi 
\]
as $\mathbb{F}_{p^m}$-representations. 
Let 
$\psi' \colon \mathbb{F}_p \hookrightarrow \overline{\mathbb{Q}}_
{\ell}^{\times}$ be any non-trivial character. 
Then, for each $\psi \in \mathbb{F}_{p^m}^{\vee} \setminus \{1\}$, 
there exists a unique element $\zeta \in 
\mathbb{F}_{p^m}^{\times}$ such that 
$\psi=\psi' \circ p_{\zeta}$.
Hence, 
we know that 
\[
 H^{N}(X_{N,\zeta},\overline{\mathbb{Q}}_{\ell})[\psi'] 
 =H^{N}(X_N,\overline{\mathbb{Q}}_{\ell})[\psi]
 \simeq \psi\] 
as $\mathbb{F}_{p^m}$-representations. 
Therefore, the required assertion follows. 
\end{proof}

Consider the fibration
\[
 \pi_{\zeta} \colon X_{N,\zeta}
 \to \mathbb{A}^{n_0};\ 
 (z_{\zeta}, (u_i)_{1 \leq i \leq N})
 \mapsto 
 ((u_{2i})_{1 \leq i \leq n_0}).
\]
Let $\mathbf{0}$ denote the origin of $\mathbb{A}^{n_0}$. 
The inverse image 
$\pi_{\zeta}^{-1}(\mathbf{0})$ has $p$ connected components. 
For $a \in \bF_p$, 
we define $Z_{\zeta}^a$ 
to be the connected component of 
$\pi_{\zeta}^{-1}(\mathbf{0})$ 
defined by $z_{\zeta}=a$. 
We know that each $Z_{\zeta}^a$ 
is isomorphic to 
the affine space of dimension $n_0$. 
Let 
\[
 \mathrm{cl} \colon 
 \mathit{CH}_{n_0}(X_{N,\zeta}) \to 
 H^{N}(X_{N,\zeta},\overline{\mathbb{Q}}_{\ell}(n_0)) 
\]
be the cycle class map. 
 
\begin{lem}\label{AS2}
\begin{enumerate}
\item 
The fibration 
$\pi_{\zeta} \colon X_{N,\zeta} \to \mathbb{A}^{n_0}$ 
is an affine bundle 
over $\mathbb{A}^{n_0} \setminus \{ \mathbf{0} \}$. 
\item
The cohomology group $H^N(X_{N,\zeta},\overline{\mathbb{Q}}_{\ell}(n_0))$ 
is generated by 
the cycle classes 
$\mathrm{cl}([Z_{\zeta}^a])$ for $a \in \bF_p$ 
with the relation 
$\sum_{a \in \bF_p} \mathrm{cl}([Z_{\zeta}^a]) =0$. 
\end{enumerate}
\end{lem}
\begin{proof}
For $1 \leq i \leq n_0$, let $U_i$ be the open subscheme of 
$\mathbb{A}^{n_0}$ defined by the condition 
that the $i$-th coordinate is not zero. 
Then $\{ U_i \}_{1 \leq i \leq n_0}$ is a covering of 
$\mathbb{A}^{n_0} \setminus \{ \mathbf{0} \}$. 
We can see that 
$\pi_{\zeta}$ is a trivial affine bundle on 
each $U_i$ by \eqref{fundd}. 
Hence the first claim follows. 

We set $U=\pi_{\zeta}^{-1}(\mathbb{A}^{n_0} \setminus \{ \mathbf{0} \})$. 
We have the long exact sequence 
\[
 H^{N-1}(U,\overline{\mathbb{Q}}_{\ell}) \to 
 H^{N}_{\pi_{\zeta}^{-1}(\mathbf{0})}
 (X_{N,\zeta},\overline{\mathbb{Q}}_{\ell}) 
 \simeq \overline{\mathbb{Q}}_{\ell}(-n_0)^{\oplus p}
 \to 
 H^{N}(X_{N,\zeta},\overline{\mathbb{Q}}_{\ell})
 \to 
 H^N(U,\overline{\mathbb{Q}}_{\ell})   
\]
and $H^N(U,\overline{\mathbb{Q}}_{\ell})
\simeq H^N(\mathbb{A}^{n_0} \setminus \{ \mathbf{0} \},\overline{\mathbb{Q}}_{\ell})=0$,
which follows from the first claim. 
Therefore, 
$H^N(X_{N,\zeta},\overline{\mathbb{Q}}_{\ell}(n_0))$ 
is generated by the cycle classes 
$\mathrm{cl}([Z_{\zeta}^a])$ for $a \in \bF_p$. 
On the other hand, we have 
$\sum_{a \in \bF_p} \mathrm{cl}([Z_{\zeta}^a]) =0$, 
since $\sum_{a \in \bF_p} [Z_{\zeta}^a] =0$ in 
$\mathit{CH}_{n_0}(X_{N,\zeta})$. 
Since 
$\dim H^N(X_{N,\zeta},\overline{\mathbb{Q}}_{\ell}(n_0))=p-1$ 
by Lemma \ref{AS0}, we obtain the claim. 
\end{proof}

\begin{cor}
The Tate conjecture in \cite[7.13]{JanMixK} 
holds for the variety $X_{N,\bF_q}$. 
\end{cor}
\begin{proof}
By Lemma \ref{AS0}, Lemma \ref{AS2} 
and the commutativity of cycle maps and pullbacks under 
$X_{N} \to X_{N,\zeta}$, 
we have \cite[7.13 Conjecture (A), (B)]{JanMixK} for $X_{N,\bF_q}$
and the equality 
\[
 H^{N}(X_{N},\overline{\mathbb{Q}}_{\ell}(n_0))^{\Gal(\ol{\bF}_q/\bF_q)} 
 =H^{N}(X_{N},\overline{\mathbb{Q}}_{\ell}(n_0)). 
\]
Then the $q$-th geometric Frobenius in $\Gal(\ol{\bF}_q/\bF_q)$ acts on 
$H^{N}(X_{N},\overline{\mathbb{Q}}_{\ell})$ by $q^{n_0}$. 
Hence \cite[7.13 Conjecture (C)]{JanMixK} for $X_{N,\bF_q}$ 
also follows. 
\end{proof}

\subsection{Action on cohomology}
In this section, we assume that $p=2$. 
Let $n \geq 4$ be an even integer. 
Let $m =\gcd (e,f)$ as in Subsection \ref{RedAff}. 
We consider the affine smooth variety $X$ of dimension $n-2$
defined by 
\[
 z^{2^m}-z=\sum_{1 \leq i \leq j \leq n-2} y_i y_j
\quad \textrm{in $\mathbb{A}^{n-1}$}.
\]
We take $\zeta_3 \in \ol{\bF}_q \setminus \{ 1 \}$ such that $\zeta_3^3=1$. 
Then, we define $u_1, \ldots , u_{n-2}$ by 
\begin{align*}
 u_{4i+1} &= \zeta_3 y_{4i+1} + \zeta_3^{-1} y_{4i+2} 
 + \sum_{j=4i+3}^{n-2} y_j, \quad 
 u_{4i+2} = \zeta_3^{-1} y_{4i+1} + \zeta_3 y_{4i+2} 
 + \sum_{j=4i+3}^{n-2} y_j, \\ 
 u_{4i+3} &= y_{4i+3} + \sum_{j=4i+5}^{n-2} y_j, \quad 
 u_{4i+4} = y_{4i+4} + \sum_{j=4i+5}^{n-2} y_j . 
\end{align*}
Then the variety $X$ is isomorphic to the affine variety $X_{n-2}$ 
defined by 
\begin{equation*}
 z^{2^m}-z=\sum_{i=1}^{n_0} u_{2i-1} u_{2i} 
\quad \textrm{in $\mathbb{A}^{n-1}$} , 
\end{equation*}
where $n_0 = (n-2)/2$. 
For $\zeta \in \bF_{2^m}^{\times}$, 
we simply write $X_{\zeta}$ for the variety 
$X_{n-2,\zeta}$, which is defined in Subsection \ref{ssec:Tate} 
where $N=n-2$. Recall that $X_{\zeta}$ 
has the defining equation 
\[
 \zeta ( z_{\zeta}^2 -z_{\zeta} )=\sum_{i=1}^{n_0} u_{2i-1} u_{2i} 
 \quad \textrm{in $\mathbb{A}^{n-1}$}. 
\]

For $a \in \bF_2$, 
we consider the other $n_0$-dimensional 
cycle $Z'^a_{\zeta}$ in $X_{\zeta}$ 
defined by 
\begin{align*}
 &u_1 =0, \quad 
 u_{4i-1}=u_{4i+2}, \quad u_{4i}=u_{4i+1} \quad 
 \textrm{for $1 \leq i \leq [(n_0 -1)/2]$},\\ 
 &u_{n-3}=u_{n-2}+1 \quad \textrm{if $e=1$}, 
 \quad 
 z_{\zeta}=a +\varepsilon_1 u_{n-2}, 
\end{align*}
where $\varepsilon_1$ is defined at \eqref{eq:ep1}. 

\begin{prop}\label{AS1}
For $\zeta \in \mathbb{F}_{2^m}^{\times}$ and 
$a \in \bF_2$, we have 
\[
 \left[Z_{\zeta}^a \right]=(-1)^{n_0}
 \left[Z'^a_{\zeta}\right] \quad 
 \textrm{in $\mathit{CH}_{n_0}(X_{\zeta})$}. 
\]
\end{prop}
\begin{proof}
We show that 
$[Z_{\zeta}^a ]-(-1)^{n_0} [Z'^a_{\zeta} ]$ 
is rationally equivalent to zero. 
For $1 \leq i \leq [(n_0 +1)/2]$, 
let $X_{\zeta,i}$ 
be the $(n_0 +1)$-dimensional closed
subvariety of $X_{\zeta}$ defined by 
\begin{align*}
 &u_{4j} =0 \quad \textrm{for $1 \leq j \leq [n_0/2]$}, \\ 
 &u_{4j-3} =0 \quad \textrm{for $1 \leq j \leq i-1$}, \quad 
 u_{4j-2} =0 \quad \textrm{for $i+1 \leq j \leq [(n_0+1)/2]$}, 
\end{align*}
and let $Z^a_{\zeta,i}$ be the $n_0$-dimensional
cycle on $X_{\zeta,i}$ 
defined by $u_{4i-3}=0$ and $z_{\zeta}=a$. 
We put $Z^a_{\zeta,0}=Z^a_{\zeta}$. 
Then we have 
\[
 \mathrm{div} (z_{\zeta}-a) 
 =[Z^a_{\zeta,i-1}]+[Z^a_{\zeta,i}]
\] 
in $\mathit{CH}_{n_0}(X_{\zeta,i})$ for $1 \leq i \leq [(n_0 +1)/2]$, 
since we have $\zeta ( z_{\zeta}^2 -z_{\zeta} )=u_{4i-3} u_{4i-2}$ 
on $X_{\zeta,i}$. 
For $1 \leq i \leq [(n_0-1)/2]$, 
let $X'_{\zeta,i}$ 
be the $(n_0 +1)$-dimensional closed 
subvariety of $X_{\zeta}$ defined by 
\begin{alignat*}{2}
 &u_1=0, \quad u_{4j-1} =u_{4j+2} ,\quad u_{4j} =u_{4j+1} \quad 
 \textrm{for $1 \leq j \leq i-1$}, \quad 
 u_{4i}=u_{4i+1}, \\ 
 &u_{4j}=0 \quad 
 \textrm{for $i+1 \leq j \leq [n_0/2]$}, \quad 
 u_{4j+1} =0 \quad 
 \textrm{for $i+1 \leq j \leq [(n_0 -1)/2]$}, 
\end{alignat*}
and let $Z'^a_{\zeta,i}$ be the $n_0$-dimensional
cycle on $X'_{\zeta,i}$ 
defined by $u_{4i-1}=u_{4i+2}$ and $z_{\zeta}=a$. 
We put $Z'^a_{\zeta,0}=Z^a_{\zeta,[(n_0+1)/2]}$. 
Then we have 
\begin{equation*}
 \mathrm{div}(z_{\zeta}-a)=
 [Z'^a_{\zeta,i-1}] + [ Z'^a_{\zeta,i} ] 
\end{equation*}
in $\mathit{CH}_{n_0}(X'_{\zeta,i})$ 
for $1 \leq i \leq [(n_0-1)/2]$, 
since we have 
\[
 \zeta ( z_{\zeta}^2 -z_{\zeta} )=u_{4i}(u_{4i-1}+u_{4i+2}) 
\]
on $X'_{\zeta,i}$. 
If $e \geq 2$, then 
$Z'^a_{\zeta,[(n_0-1)/2]}=Z'^a_{\zeta}$, and 
the claim follows. 
Assume that $e=1$. Then $m=1$. 
Let $X''_{\zeta}$ 
be the $(n_0 +1)$-dimensional closed 
subvariety of $X_{\zeta}$ defined by 
\[
 u_1=0, \quad u_{4j-1} =u_{4j+2} ,\quad u_{4j} =u_{4j+1} \quad 
 \textrm{for $1 \leq j \leq (n_0 -2)/2$}. 
\]
Then we have 
\[
  \mathrm{div}(z_{\zeta}-u_{n-2}-a)=
 [Z'^a_{\zeta,[(n_0-1)/2]}] + [ Z'^a_{\zeta} ] 
\]
in $\mathit{CH}_{n_0}(X''_{\zeta})$, 
since we have 
\[
 (z-u_{n-2})(z-u_{n-2}-1)=u_{n-2}(u_{n-3}+u_{n-2}+1) 
\] 
on $X''_{\zeta}$. 
Therefore, we obtain the claim. 
\end{proof}

\begin{cor}\label{AS3}
Assume that $n \geq 4$.
Let $g$ be the automorphism of $X$ defined by 
\[
 (z,(y_i)_{1 \leq i \leq n-2}) 
 \mapsto 
 \left( z +\varepsilon_1 (y_{n-2} +1 ), 
 \sum_{i=1}^{n-3} y_i  
 +\varepsilon_1, 
 ( y_{i-1} + y_{n-2} +\varepsilon_1 )_{2 \leq i \leq n-2} \right) . 
\]
Then, $g^\ast $ acts on $H^{n-2}(X,\overline{\mathbb{Q}}_{\ell})$ by 
$-1$. 
\end{cor}
\begin{proof}
Note that $g$ induces an automorphism of $X_{\zeta}$. 
The condition of $Z^0_{\zeta} \subset X_{\zeta}$ is equivalent to 
\begin{align*}
 &y_{4i -3} + \zeta_3^{-1} y_{4i-2} + \zeta_3 y_{4i-1} =0 
 \quad \textrm{for $1 \leq i \leq [n_0/2]$,} \\ 
 &y_{4i} + \zeta_3 y_{4i+1} + \zeta_3^{-1} y_{4i+2} =0 
 \quad \textrm{for $1 \leq i \leq [(n_0 -1)/2]$,} \\ 
 &\zeta_3^{-1} y_{n-3} + \zeta_3 y_{n-2} =0 \quad 
 \textrm{if $e \geq 2$,} \quad 
 y_{n-2} =0 \quad \textrm{if $e =1$}, \quad 
 z_{\zeta} =0. 
\end{align*}
For $a \in \bF_2$, 
the condition of $Z'^a_{\zeta} \subset X_{\zeta}$ is equivalent to 
\begin{align*}
 &\zeta_3 y_1 + \zeta_3^{-1} y_2  + \sum_{j=3}^{n-2} y_j =0, \\ 
 &y_{4i -1} + \zeta_3 y_{4i} + \zeta_3^{-1} y_{4i+1} =0, 
 \quad 
 y_{4i} + \zeta_3^{-1} y_{4i+1} + \zeta_3 y_{4i+2} =0 
 \quad \textrm{for $1 \leq i \leq [(n_0 -1)/2]$,} \\ 
 & y_{n-3} =y_{n-2} +1 \quad \textrm{if $e =1$}, \quad 
 z_{\zeta} =a +\varepsilon_1 y_{n-2}. 
\end{align*}
Using the above, we can check that 
\[
 g^{-1}(Z^0_{\zeta}) = 
 \begin{cases}
 Z'^1_{\zeta} & \textrm{if $e=1$}, \\
 Z'^0_{\zeta} & \textrm{otherwise}.
\end{cases}
\]
Therefore, we obtain 
\[
 g^\ast \bigl( \mathrm{cl} ([Z^0_{\zeta} ]) \bigr) 
 =(-1)^{n_0 +1} \bigl( \mathrm{cl} ([Z'^0_{\zeta} ]) \bigr) 
 =-\mathrm{cl} ([Z^0_{\zeta} ]) 
\] 
in $H^{n-2}(X_{\zeta},\overline{\mathbb{Q}}_{\ell}(n_0))$ 
using 
Lemma \ref{AS2} 
and Proposition \ref{AS1}. 
Hence, 
the claim follows from Lemma \ref{AS0}
and Lemma \ref{AS2}. 
\end{proof}

\section{Explicit LLC and LJLC}\label{sec:explicit}
\subsection{Galois representations}\label{CohGal}
Let $X$ be the affine 
smooth variety over $k^{\mathrm{ac}}$ 
defined by \eqref{vareq}. 
We define an action of $Q \rtimes \bZ$ on $X$ 
similarly to \eqref{QZactXr}.

We choose an isomorphism 
$\iota \colon \overline{\mathbb{Q}}_{\ell} \simeq \mathbb{C}$.
Let $q^{1/2} \in  \overline{\mathbb{Q}}_{\ell}$ be the 
$2$-nd root of $q$ such that 
$\iota(q^{1/2})>0$. 
For a rational number $r \in 2^{-1}\mathbb{Z}$, 
let $\overline{\mathbb{Q}}_{\ell}(r)$ be the unramified representation of 
$\Gal(k^{\mathrm{ac}}/k)$ of degree $1$, 
on which the geometric Frobenius 
$\mathrm{Frob}_q$ acts as scalar multiplication by 
$q^{-r}$.
We simply write $Q$ for the subgroup 
$Q \times \{0\} \subset Q \rtimes \mathbb{Z}$. 
We consider the morphisms 
\begin{align*}
 \Phi & \colon \mathbb{A}_{k^{\mathrm{ac}}}^{n-1} 
 \to \mathbb{A}_{k^{\mathrm{ac}}}^1 ;\ 
 (y,(y_i)_{1 \leq i \leq n-2} ) 
 \mapsto y^{p^e +1} 
 - \frac{1}{n'} \sum_{1 \leq i \leq j \leq n-2} y_i y_j , \\ 
 h_m & \colon \mathbb{A}_{k^{\mathrm{ac}}}^1 \to \mathbb{A}_{k^{\mathrm{ac}}}^1; 
 z \mapsto z^{p^m} -z . 
\end{align*}
Then we have a cartesian diagram 
\[
 \xymatrix{
 X \ar@{->}[r] \ar@{->}[d] & 
 \mathbb{A}_{k^{\mathrm{ac}}}^{n-1} 
 \ar@{->}^-{\Phi}[d] \\ 
 \mathbb{A}_{k^{\mathrm{ac}}}^1 \ar@{->}^-{h_m}[r] & 
 \mathbb{A}_{k^{\mathrm{ac}}}^1 . 
 }
\]
Using the proper base change theorem for the above cartesian diagram, 
we have a decomposition 
\begin{equation}\label{HXdec}
 H^{n-1}_{\mathrm{c}}(X,\overline{\mathbb{Q}}_{\ell}) 
 \simeq 
 \bigoplus_{\psi \in \bF_{p^m}^{\vee} \backslash \{1\}} 
 H_{\mathrm{c}}^{n-1} (\mathbb{A}_{k^{\mathrm{ac}}}^{n-1} ,
 \mathcal{L}_{\psi} (\Phi) ) , 
\end{equation}
since 
${h_m}_* \ol{\bQ}_{\ell} \simeq \bigoplus_{\psi \in \bF_{p^m}^{\vee}} 
 \mathcal{L}_{\psi}$ 
and 
$H_{\mathrm{c}}^{n-1} (\mathbb{A}_{k^{\mathrm{ac}}}^{n-1} ,
 \ol{\bQ}_{\ell} ) =0$. 
The decomposition \eqref{HXdec} is stable under the action of 
$Q \rtimes \bZ$, 
since $\bF_{p^m} \simeq \{ g(1,0,c) \mid c \in \bF_{p^m} \}$ 
in the center of $Q \rtimes \bZ$ 
acts on each direct summand 
$H_{\mathrm{c}}^{n-1} (\mathbb{A}_{k^{\mathrm{ac}}}^{n-1} ,
 \mathcal{L}_{\psi} (\Phi) )$ in \eqref{HXdec} 
by $\psi$. 
We put 
\[
 \tau_{\psi,n} = 
 H_{\mathrm{c}}^{n-1} (\mathbb{A}_{k^{\mathrm{ac}}}^{n-1} ,
 \mathcal{L}_{\psi} (\Phi) ) \biggl( \frac{n-1}{2} \biggr) 
\]
as a $Q \rtimes \bZ$-representation 
for each 
$\psi \in \bF_{p^m}^{\vee} \backslash \{1\}$. 
We write 
$\tau^0_{r,\psi}$ 
for the inflation of 
$\tau_{\psi,n}$ by $\Theta_r$ in \eqref{hom}.

\subsection{Correspondence}\label{ConsRep}
\begin{defn}\label{def:ss}
We say that 
an irreducible supercuspidal representation 
of $\iGL_n (K)$ 
is simple supercuspidal 
if its exponential Swan conductor is one. 
\end{defn}

\begin{rem}
Definition \ref{def:ss} is compatible with 
\cite[Definition 1.1]{ITsimpJL} by 
\cite[Proposition 1.3]{ITsimpJL}. 
The word 
``simple supercuspidal'' comes from 
\cite{GRAinv}. 
Our ``simple supercuspidal'' representations 
are called ``epipelagic'' in \cite{BHLepi} 
after \cite{RYEinv}. 
\end{rem}

We define $\psi_0 \in \bF_p^{\vee}$ by 
$\iota (\psi_0 (1) )=\exp (2\pi \sqrt{-1}/p)$. 
We put 
$\psi_0' =\psi_0 \circ \Tr_{\bF_{p^m}/\bF_p}$. 
We take an additive character 
$\psi_K \colon K \to \ol{\bQ}_{\ell}^{\times}$ 
such that 
$\psi_K (x) =\psi_0'(\bar{x})$ 
for $x \in \cO_K$. 
In the following,  for each triple 
$(\zeta,\chi,c) \in \mu_{q-1}(K) \times 
(k^{\times})^{\vee} \times \overline{\mathbb{Q}}_{\ell}^{\times}$, 
we define a $\textit{GL}_n(K)$-representation 
$\pi_{\zeta,\chi,c}$, a $D^{\times}$-representation $\rho_{\zeta,\chi,c}$ and 
a $W_K$-representation $\tau_{\zeta,\chi,c}$.

We use notations in Subsection \ref{ConstAff}, 
replacing $r \in \mu_{q-1}(K)$ with 
$\zeta \in \mu_{q-1}(K)$. 
We have the $K$-algebra embeddings 
\begin{align*}
 L_{\zeta} \to M_n (K) ;\ \varphi_{\zeta} \mapsto \varphi_{M,\zeta}, \quad 
 L_{\zeta} \to D ;\ \varphi_{\zeta} \mapsto \varphi_{D,\zeta}. 
\end{align*}
Set $\varphi_{\zeta,n} =n' \varphi_{\zeta}$. 
Let 
$\Lambda_{\zeta,\chi,c} \colon L_{\zeta}^{\times} U_{\mathfrak{I}}^1 
 \to \overline{\mathbb{Q}}_{\ell}^{\times}$ 
 be the character defined by
\begin{align*}
\Lambda_{\zeta,\chi,c} (\varphi_{\zeta})& =
(-1)^{n-1}c, \quad 
\Lambda_{\zeta,\chi,c} (x) =\chi (\bar{x}) \quad 
\textrm{for $x \in \mathcal{O}_{K} ^{\times}$},\\ 
 \Lambda_{\zeta,\chi,c} (x) & =
(\psi_K \circ 
 \mathrm{tr} )(\varphi_{\zeta,n}^{-1}(x-1)) \quad 
\textrm{for $x \in U_{\mathfrak{I}}^1$}.
\end{align*}
 We put 
\[
 \pi_{\zeta,\chi,c} = 
 \mathrm{c\mathchar`-Ind}_{L_{\zeta}^{\times} U_{\mathfrak{I}}^1}^{\mathit{GL}_n(K)} 
 \Lambda_{\zeta,\chi,c}. 
\]
Then, $\pi_{\zeta,\chi,c}$ 
is a simple supercuspidal representation of $\mathit{GL}_n (K)$, and 
every simple supercuspidal representation 
is isomorphic to 
$\pi_{\zeta,\chi,c}$ for a uniquely determined 
$(\zeta,\chi,c) \in \mu_{q-1}(K) \times (k^{\times})^{\vee} 
 \times \overline{\mathbb{Q}}_{\ell}^{\times}$ 
(\cf \cite[2.1, 2.2]{BHLepi}). 

Let  
$\theta_{\zeta,\chi,c} \colon L_{\zeta}^{\times} U_D ^1 
 \to \overline{\mathbb{Q}}_{\ell}^{\times}$ 
 be the character defined by 
 \begin{align*}
\theta_{\zeta,\chi,c} (\varphi_{\zeta})& =c, \quad 
\theta_{\zeta,\chi,c} (x) =\chi (\bar{x}) \quad \textrm{for 
$x \in \mathcal{O}_{K}^{\times}$}, \\ 
\theta_{\zeta,\chi,c} (d) &=
\left(\psi_K \circ 
\mathrm{Trd}_{D /K}\right)(\varphi_{\zeta,n}^{-1}(d-1)) 
\quad \textrm{for $d \in U_D ^1$}.
\end{align*}
We put 
\[
 \rho_{\zeta, \chi,c} = 
 \mathrm{Ind}_{L_{\zeta}^{\times} U_D ^1}^{D^{\times}} 
\theta_{\zeta, \chi,c}. 
\]
The isomorphism class of 
this representation does not depend 
on the choice of the 
embedding $L_{\zeta} \hookrightarrow D$.

Recall that 
$\varphi_{\zeta}'=\varphi_{\zeta}^{p^e}$ and 
$E_{\zeta} =K(\varphi_{\zeta}')$. 
Let $\phi_c \colon W_{E_{\zeta}} \to 
\overline{\mathbb{Q}}_{\ell}^{\times}$ 
be the character defined by 
$\phi_c(\sigma)=c^{n_{\sigma}}$.
Let 
$\mathrm{Frob}_p \colon k^{\times} \to k^{\times}$ be the map defined by 
$x \mapsto x^{p^{-1}}$
for $x \in k^{\times}$.
We consider the composite 
\[
 \nu_{\zeta} \colon 
 W_{E_{\zeta}}^{\mathrm{ab}} 
 \xrightarrow{\mathrm{Art}_{E_{\zeta}}^{-1}} E_{\zeta}^{\times} \to 
 \mathcal{O}_{E_{\zeta}}^{\times} 
 \xrightarrow{\mathrm{can.}} k^{\times} \xrightarrow{\mathrm{Frob}_p^e} k^{\times}, 
\]
where the second homomorphism is given by 
$E_{\zeta}^{\times}  \to \mathcal{O}_{E_{\zeta}}^{\times};\ x \mapsto 
 x {\varphi_{\zeta}'}^{-v_{E_{\zeta}}(x)}$. 
We simply 
write $\tau_{\zeta}^0$
for 
$\tau_{\zeta,\psi_0'}^0$.
We set 
\[
 \tau_{\zeta,\chi,c}^0=\tau_{\zeta}^0 \otimes 
(\chi \circ \nu_{\zeta}) \otimes \phi_c, \quad 
 \tau_{\zeta,\chi,c}=
 \mathrm{Ind}_{W_{E_{\zeta}}}^{W_K} \tau_{\zeta,\chi,c}^0. 
\]
We see that 
$\tau_{\zeta,\chi,c}^0$ is primitive by 
\cite[3.2 Proposition]{BHLepi} and \cite{ITlgsw1}. 

The following theorem 
follows from \cite{ITlgsw1} and \cite{ITsimpJL}. 

\begin{thm}\label{expJL}
Let $\mathrm{LL}$ and $\mathrm{JL}$ denote the local 
Langlands correspondence and the local 
Jacquet--Langlands correspondence for $\mathit{GL}_n(K)$ respectively. 
For $\zeta \in \mu_{q-1} (K)$, 
$\chi \in (k^{\times})^{\vee}$ and 
$c \in \overline{\mathbb{Q}}_{\ell}^{\times}$, 
we have 
$\mathrm{LL}(\pi_{\zeta,\chi,c})=\tau_{\zeta,\chi,c}$ and 
$\mathrm{JL} (\rho_{\zeta, \chi,c}) =\pi_{\zeta,\chi,c}$. 
\end{thm}

\begin{defn}
We say that 
a smooth irreducible representation 
of $\iGL_n (K)$ 
is essentially simple supercuspidal 
if it is a character twist of a simple supercuspidal representation. 
\end{defn}
Let $\omega \colon K^{\times} \to \ol{\bQ}_{\ell}^{\times}$ 
be a smooth character. 
We put 
\[
 \pi_{\zeta,\chi,c,\omega} =\pi_{\zeta,\chi,c} \otimes (\omega \circ \det), 
 \quad 
 \rho_{\zeta,\chi,c,\omega} =\rho_{\zeta,\chi,c} \otimes (\omega \circ \Nrd_{D/K}), 
 \quad 
 \tau_{\zeta,\chi,c,\omega} =\tau_{\zeta,\chi,c} \otimes 
 (\omega \circ \Art_K^{-1}), 
\]
and 
\begin{align*}
 \Lambda_{\zeta,\chi,c,\omega} &=\Lambda_{\zeta,\chi,c} 
 \otimes (\omega \circ \det |_{L_{\zeta}^{\times} U_{\fI}^1}), 
 \quad
 \theta_{\zeta,\chi,c,\omega} =\theta_{\zeta,\chi,c} 
 \otimes (\omega \circ \Nrd_{D/K}|_{L_{\zeta}^{\times} U_D^1}), 
 \\
 \tau_{\zeta,\chi,c,\omega}^0 &=\tau_{\zeta,\chi,c}^0 \otimes 
 (\omega \circ \Nr_{E_{\zeta}/K} \circ \Art_{E_{\zeta}}^{-1}). 
\end{align*}
Then we have 
\[
 \pi_{\zeta, \chi,c,\omega} = 
 \mathrm{c\mathchar`-Ind}_{L_{\zeta}^{\times} U_{\fI}^1}^{\iGL_n(K)} 
 \Lambda_{\zeta, \chi,c,\omega}, \quad 
 \rho_{\zeta, \chi,c,\omega} = 
 \mathrm{Ind}_{L_{\zeta}^{\times} U_D ^1}^{D^{\times}} 
 \theta_{\zeta, \chi,c,\omega}, \quad 
 \tau_{\zeta,\chi,c,\omega}=
 \mathrm{Ind}_{W_{E_{\zeta}}}^{W_K} \tau_{\zeta,\chi,c,\omega}^0 . 
\]

\begin{cor}\label{expJLes}
We have 
$\mathrm{LL}(\pi_{\zeta,\chi,c,\omega})=\tau_{\zeta,\chi,c,\omega}$ and 
$\mathrm{JL} (\rho_{\zeta, \chi,c,\omega}) =\pi_{\zeta,\chi,c,\omega}$. 
\end{cor}
\begin{proof}
This follows from Theorem \ref{expJL}, 
because $\mathrm{LL}$ and $\mathrm{JL}$ are compatible with character twists. 
\end{proof}

\section{Geometric realization}\label{GeomRela}
Recall that $n_1 =\gcd (n,p^m -1)$. 
We fix 
$s \in \mu_{\frac{n_1(q-1)}{p^m-1}}(K)$. 
We take an element 
$r \in \mu_{q-1}(K)$ 
such that 
$r^{\frac{p^m-1}{n_1}}=s$. 
We put 
\[
 H_{\fX_r} =H_{\mathrm{c}}^{n-1}(\ol{\fX}_r,\overline{\mathbb{Q}}_{\ell})\biggl( \frac{n-1}{2} \biggr)
\]
as $H_r$-representations. 
\begin{lem}\label{lem:indep}
The isomorphism class of 
$\mathrm{c\mathchar`-Ind}_{H_r}^G 
 H_{\fX_r}$ depends only on $s$. 
\end{lem}
\begin{proof}
Assume that $r,r' \in \mu_{q-1}(K)$ 
satisfy 
\[
 r^{\frac{p^m-1}{n_1}}=r'^{\frac{p^m-1}{n_1}}=s. 
\]
Then we have $L_r =L_{r'}$. 
Hence, there is $(g,d) \in (\iGL_n (K) \times D^{\times})^0$ 
such that 
$\xi_r (g,d)=\xi_{r'}$ by Lemma \ref{CMtrans}. 
Then we have $\fX_r (g,d)=\fX_{r'}$. 
Threfore we obtain the calim. 
\end{proof}
We put 
\[
 \Pi_s= \mathrm{c\mathchar`-Ind}_{H_r}^G 
 H_{\fX_r}. 
\]
For simplicity, we write
$G_1$ and $G_2$ for $\mathit{GL}_n(K)$ and $D^{\times} \times W_K$ 
respectively, and consider them as subgroups of $G$. 
We put 
\[
 H =\{ g \in U_{\fI}^1 \mid \det (g)=1 \}. 
\] 
We have $H=H_r \cap G_1$ by Proposition \ref{gda}.
Let $\overline{H}_r$ be the 
image of $H_r$ in $G/G_1\simeq G_2$. 

Let 
$a \in \mu_{q-1}(K)$. 
We define a character 
$\Lambda_{r}^a \colon U_{\mathfrak{I}}^1 \to 
\overline{\mathbb{Q}}_{\ell}^{\times}$ by 
\[
 \Lambda_{r}^a (x) = 
 (\psi_K \circ \tr)((a \varphi_{r,n})^{-1}(x-1)) \quad 
 \textrm{ for $x \in U_{\mathfrak{I}}^1$}. 
\]
Let $\pi$ be a smooth irreducible 
representation of $\iGL_n (K)$. 

\begin{lem}\label{lf} 
If $\pi$ is not essentially simple supercuspidal, 
then we have 
$\Hom_H (\Lambda_r^a ,\pi)=0$. 
Further, we have 
\[
 \dim \mathrm{Hom}_{H }
 ( \Lambda_{r}^a ,\pi_{\zeta,\chi,c,\omega} )=
 \begin{cases}
  1\quad  & \textrm{if $a^n r=\zeta$},  \\
  0\quad  &  \mathrm{otherwise}.
 \end{cases}
\]
\end{lem}
\begin{proof}
We assume that $\Hom_H (\Lambda_r^a ,\pi) \neq 0$, 
and show that $\pi$ is essentially simple supercuspidal. 
Let $\omega_{\pi}$ be the central character of $\pi$. 
Then $\omega_{\pi}$ is trivial on 
$K^{\times} \cap H$ by $\Hom_H (\Lambda_r^a ,\pi) \neq 0$. 
Hence, we may assume that 
$\omega_{\pi}$ is trivial on 
$K^{\times} \cap U_{\fI}^1$, 
changing $\pi$ by a character twist. 
Then, there is a character 
$\Lambda_{r,\omega_{\pi}}^a \colon K^{\times} U_{\fI}^1 \to \ol{\bQ}_{\ell}^{\times}$ 
such that 
\[
 \Lambda_{r,\omega_{\pi}}^a|_{U_{\fI}^1} =\Lambda_r^a , \quad 
 \Lambda_{r,\omega_{\pi}}^a|_{K^{\times}}=\omega_{\pi} . 
\]
Then we have 
\begin{equation}\label{HomHHU}
 \mathrm{Hom}_{H }
 ( \Lambda_{r}^a ,\pi ) \simeq 
 \mathrm{Hom}_{K^{\times} H}
 ( \Lambda_{r,\omega_{\pi}}^a ,\pi ) \simeq 
 \mathrm{Hom}_{K^{\times} U_{\fI}^1} 
 \Bigl( \Ind_{K^{\times} H}^{K^{\times} U_{\fI}^1} 
 (\Lambda_{r,\omega_{\pi}}^a |_{K^{\times} H} ) ,\pi \Bigr) 
\end{equation}
by Frobenius reciprocity, since 
$K^{\times} U_{\mathfrak{I}}^1/(K^{\times} H)$ 
is compact. 
We have the natural isomorphism 
\begin{equation}\label{UI1iso}
 K^{\times} U_{\fI}^1 / (K^{\times} H) 
 \xrightarrow{\det}
 (K^{\times})^n U_K^1 / (K^{\times})^n 
 \simeq 
 U_K^1 / (U_K^1)^n.  
\end{equation}
For a smooth character 
$\phi$ of $U_K^1 / (U_K^1)^n$, 
let $\phi'$ denote the character of $K^{\times} U_{\fI}^1$ 
obtained by $\phi$ and the isomorphism \eqref{UI1iso}. 
We have a natural isomorphism 
\begin{equation}\label{IndHU}
 \Ind_{K^{\times} H}^{K^{\times} U_{\fI}^1} 
 (\Lambda_{r,\omega_{\pi}}^a |_{K^{\times} H} )
 \simeq 
 \bigoplus_{\phi \in (U_K^1 / (U_K^1)^n)^{\vee}} 
 \Lambda_{r,\omega_{\pi}}^a \otimes \phi' . 
\end{equation}
Let $\phi$ be a smooth character of $U_K^1 / (U_K^1)^n$, 
and regard it as a character of $U_K^1$. 
We extend $\phi$ to a character $\tilde{\phi}$ of $K^{\times}$ 
such that $\tilde{\phi} (\varpi)=1$ and 
$\tilde{\phi}$ is trivial on $\mu_{q-1} (K)$. 
We have 
\begin{equation}\label{HomUG1}
 \Hom_{K^{\times} U_{\fI}^1} 
 ( \Lambda_{r,\omega_{\pi}}^a \otimes \phi' ,\pi ) \simeq 
 \Hom_{G_1} 
 \Bigl( \bigl( \mathrm{c\mathchar`-Ind}_{K^{\times} U_{\fI}^1}^{G_1} 
 \Lambda_{r,\omega_{\pi}}^a \bigr) \otimes \tilde{\phi} ,\pi \Bigr) 
\end{equation}
by Frobenius reciprocity. 
We take $\chi' \in (k^{\times})^{\vee}$ such that 
$\chi'(\bar{x})=\omega_{\pi} (x)$ for $x \in \mu_{q-1} (K)$. 
For $c' \in \ol{\bQ}_{\ell}^{\times}$, 
we define the character 
$\Lambda^a_{r,\chi',c'} \colon L_r^{\times}U_{\mathfrak{I}}^1
\to \overline{\mathbb{Q}}_{\ell}^{\times}$
by 
\[
 \Lambda^a_{r,\chi',c'} |_{U_{\mathfrak{I}}^1}=\Lambda_{r}^a, \quad 
 \Lambda^a_{r,\chi',c'} (\varphi_{M,r})=c', \quad
 \Lambda^a_{r,\chi',c'} (x)=\chi'(\bar{x}) \quad
 \textrm{ for $x \in \mu_{q-1}(K)$}. 
\]
We put 
\[
 \pi^a_{r,\chi',c'}= 
\mathrm{c\mathchar`-Ind}_{L_r^{\times} U_{\mathfrak{I}}^1}^{G_1} 
 \Lambda^a_{r,\chi',c'}. 
\]
Then we have 
\begin{equation}\label{IndLa}
 \mathrm{c\mathchar`-Ind}_{K^{\times} U_{\fI}^1}^{G_1} 
 \Lambda_{r,\omega_{\pi}}^a \simeq 
 \bigoplus_{c' \in \ol{\bQ}_{\ell}^{\times}} 
 \pi^a_{r,\chi',c'}. 
\end{equation}
Note that 
\begin{equation}\label{piapi}
 \pi^a_{r,\chi',c'}
 \simeq 
 \pi_{a^n r,\chi',\chi' (a) c'}
\end{equation}
by the constructions. 
Then we see that $\pi$ is simple supercuspidal by 
\eqref{HomHHU}, \eqref{IndHU}, \eqref{HomUG1}, \eqref{IndLa}, 
\eqref{piapi} and the assumption $\Hom_H (\Lambda_r^a ,\pi) \neq 0$. 

Let $\chi' \in (k^{\times})^{\vee}$. 
We use the same notations as above for such $\chi'$. 
For an irreducible supercuspidal representation $\pi$ of $G_1$, 
we write $\mathrm{a}(\pi)$ 
for its Artin conductor exponent as in \cite[1.2]{BHLepi}. 
We have $\mathrm{a}(\pi^a_{r,\chi',c'})=n+1$ by \eqref{piapi}. 
Hence, 
if $\phi \neq 1$, 
we have 
\[
 \mathrm{a}(\pi^a_{r,\chi',c'}\otimes \tilde{\phi})
 =n\mathrm{a}(\tilde{\phi}) \geq 2n 
\]
by 
$\mathrm{a}(\tilde{\phi}) \geq 2$ and 
\cite[6.5 Theorem (ii)]{BHKLocRS}. 
Therefore, we obtain 
\[
 \dim
 \mathrm{Hom}_{G_1} 
 (\pi^a_{r,\chi',c'} \otimes \tilde{\phi},\pi_{\zeta,\chi,c} )= 
 \begin{cases}
 1 \quad & \textrm{if $\phi=1$, $a^n r=\zeta$ and $\chi' (a) c'= c$},\\
 0 \quad & \textrm{otherwise} 
 \end{cases}
\]
by \eqref{piapi} and \cite[2.2]{BHLepi}. 
To show the second claim, 
we may assume that $\omega=1$. 
Hence, we obtain the second claim by the above discussion, 
using that $\omega_{\pi_{\zeta,\chi,c}}$ is trivial on 
$U_K^1$. 
\end{proof}

\begin{prop}\label{pro} 
\begin{enumerate}
\item\label{enu:dHHpi}
If $\pi$ is not essentially simple supercuspidal, 
we have 
$\mathrm{Hom}_{H} (H_{\fX_r},\pi )=0$. 
Further, we have 
\begin{equation*}\label{n_qq}
 \dim\mathrm{Hom}_{H}
 (H_{\fX_r},\pi_{\zeta,\chi,c,\omega} )=
 \begin{cases}
 p^e n_1 \quad  &\textrm{if $\zeta^{\frac{p^m-1}{n_1}}=s$},\\
 0 \quad &  \textrm{otherwise}.
 \end{cases} 
\end{equation*}
\item
We have $L_r^{\times}U_D^1 \times 
W_{E_r} \subset \overline{H}_r$ and 
an injective homomorphism 
\[
 \theta_{r,\chi,c,\omega} \otimes \tau_{r,\chi,c,\omega}^0 \hookrightarrow
 \mathrm{Hom}_{H} (H_{\fX_r},\pi_{r,\chi,c,\omega} ) 
\]
as $L_r^{\times}U_D^1 \times W_{E_r}$-representations. 
\end{enumerate}
\end{prop}
\begin{proof}
By \eqref{HXdec}, we have a decomposition 
\begin{equation}\label{HfXtau}
 H_{\fX_r} 
 \simeq 
 \bigoplus_{\psi \in \bF_{p^m}^{\vee} \backslash \{1\}} 
 \tau_{\psi,n} 
\end{equation}
as representations of $Q \rtimes \bZ$. 
By Proposition \ref{gda} and \eqref{HfXtau}, 
we have 
\begin{equation}\label{fn}
 H_{\fX_r} 
 \simeq 
 \bigoplus_{a \in \mu_{p^m-1}(K)} (\Lambda_{r}^{-a})^{\oplus p^e}
\end{equation}
as $H$-representations. 
By \eqref{fn}, we have 
\begin{equation*}\label{as}
 \mathrm{Hom}_{H} (H_{\fX_r},\pi_{\zeta,\chi,c,\omega} ) 
 \simeq 
 \bigoplus_{a \in \mu_{p^m-1}(K),\, (-a)^n r =\zeta} 
 \mathrm{Hom}_{H } (\Lambda_{r}^{-a} , \pi_{\zeta,\chi,c,\omega} 
 )^{\oplus p^e}. 
\end{equation*}
The cardinality of 
\[
 \{ a \in \mu_{p^m-1}(K) \mid (-a)^n r=\zeta \} 
\]
equals $n_1$ if $\zeta^{\frac{p^m-1}{n_1}} = s$ and zero
otherwise. 
Hence the first claim follows from Lemma \ref{lf}. 

We prove the second claim. 
We consider the element 
\[
 (\varphi_{D,r},1) \in L_r^{\times}U_D^1 \times W_{E_r} \subset G_2 
\]
and its lifting 
$\mathbf{g}_r \in G$ in \eqref{g_r} 
with respect to $G \to G_2$. 
We have 
$\mathbf{g}_r \in H_r$ 
by Proposition \ref{frob} \ref{enu:actgr}. 
The element $(\varphi_{D,r},1)$
acts on 
$\theta_{r,\chi,c,\omega} \otimes \tau_{r,\chi,c,\omega}^0$
as scalar multiplication by 
$c \omega((-1)^{n-1} \varpi_r )$, 
because $\Nrd_{D/K}(\varphi_{D,r})=(-1)^{n-1} \varpi_r$. 
By Proposition \ref{frob} \ref{enu:detgr}, Corollary \ref{AS3} and 
\cite[Proposition 4.2.3]{ITsimptame}, 
the element $\mathbf{g}_r$ 
acts on $\mathrm{Hom}_{H}(H_{\fX_r},\pi_{r,\chi,c,\omega})$ 
as scalar multiplication by 
$c \omega((-1)^{n-1} \varpi_r )$.

Let $zd \in \mathcal{O}_K^{\times}U_D^1$ 
with $z \in \mu_{q-1}(K)$
and $d \in U_D^1$.
Let 
$g=(a_{i,j})_{1 \leq i,j \leq n} \in U_\mathfrak{I}^1$
be the element defined by 
$a_{1,1}=\Nrd_{D/K}(d)$, $a_{i,i}=1$ for $2 \leq i \leq n$ and $a_{i,j}=0$ if 
$i \neq j$.
We have $\mathrm{det}(g)=\Nrd_{D/K}(d)$ and 
$(zg,zd,1) \in H_r$. 
The element 
$(zd,1) \in L_r^{\times}U_D^1 \times W_{E_r}$ 
acts on 
$\theta_{r,\chi,c,\omega} \otimes \tau_{r,\chi,c,\omega}^0$ 
as scalar multiplication by
\[
 \chi(\bar{z}) \theta_{r,\chi,c}(d) \omega(\Nrd_{D/K}(zd)). 
\]
We have the subspace 
\begin{equation}\label{iss}
 \Hom_{H} (\tau_{\psi_0'^{-1},n} ,\pi_{r,\chi,c,\omega} ) 
 \subset  \mathrm{Hom}_{H}
 (H_{\fX_r},\pi_{r,\chi,c,\omega} ) 
\end{equation}
by the decomposition \eqref{HfXtau}. 
By Remark \ref{Ktri}, Proposition \ref{gda} and 
\cite[Propositions 4.2.1 and 4.5.1]{ITsimptame}, 
the element $(zg,zd,1)$ 
acts on the subspace in \eqref{iss} 
as scalar multiplication by
\[
 \chi(\bar{z}) \theta_{r,\chi,c} (d) \omega(\det (zg)). 
\]

Let $\sigma \in W_{E_r}$ 
such that $n_{\sigma}=1$.
We take $\mathbf{g}_{\sigma}$ as in \eqref{gsig}. 
By Proposition \ref{LT}, 
the element 
$\mathbf{g}_{\sigma}$ acts on the subspace
\eqref{iss} by 
\[
 \chi (\bar{b}_{\sigma}) \tau_{r,\psi_0'}^0(\sigma) \omega (\det(g_{\sigma})). 
\] 
On the other hand,  
the element 
$(\varphi_{D,r}^{-1},\sigma) \in L_r^{\times}U_D^1 \times W_{E_r}$ acts on 
$\theta_{r,\chi,c,\omega} \otimes \tau_{r,\chi,c,\omega}^0$ 
by 
\[
 (\chi \circ \nu_r )(\sigma) \tau_{r,\psi_0'}^0(\sigma) 
 \omega (\Nr_{E_r/K}(u_{\sigma})). 
\] 
Hence, the required assertion follows from 
$\nu_r(\sigma)=\bar{b}_{\sigma}$ and 
$\Nr_{E_r/K}(u_{\sigma})=\det (g_{\sigma})$. 
\end{proof}

\begin{prop}\label{map} 
If 
$\pi$ is not essentially simple supercuspidal, 
then we have 
$\mathrm{Hom}_{\iGL_n (K)} ( \Pi_s ,\pi )=0$. 
Further, we have 
\[
 \mathrm{Hom}_{\mathit{GL}_n(K)} (\Pi_s,\pi_{\zeta,\chi,c,\omega} )
 \simeq 
 \begin{cases}
 \rho_{\zeta,\chi,c,\omega} \otimes \tau_{\zeta,\chi,c,\omega} \quad  & 
 \textrm{if $\zeta^{\frac{p^m-1}{n_1}}=s$}, \\
 0 \quad & \textrm{otherwise}
 \end{cases}
\]
as $D^{\times} \times W_K$-representations.
\end{prop}
\begin{proof}
For $g \in H_r \backslash G/G_1$, 
we choose an element $\tilde{g} \in G_2$ 
whose image in $\overline{H}_r \backslash G_2$ equals $g$ 
under the natural isomorphism 
$H_r \backslash G/G_1 \simeq \overline{H}_r \backslash G_2$. 
We put 
$H^{\tilde{g}}=\tilde{g}^{-1}H  \tilde{g}$.  
Let $H_{\fX_r}^{\tilde{g}}$ denote the representation of 
$H^{\tilde{g}}$ which is the conjugate of $H_{\fX_r}$ 
by $\tilde{g}$. 
Then, we have 
\begin{equation}\label{g1}
 \Pi_s|_{G_1} \simeq \bigoplus_{g \in H_r \backslash G/G_1} 
 \mathrm{c\mathchar`-Ind}_{H^{\tilde{g}}}^{G_1}
 H_{\fX_r}^{\tilde{g}} 
 \simeq 
 \bigoplus_{\overline{H}_r \backslash G_2} 
 \mathrm{c\mathchar`-Ind}_{H}^{G_1}H_{\fX_r} 
\end{equation}
as $G_1$-representations by Mackey's decomposition theorem, 
since we have $H^{\tilde{g}}=H$ and 
$H_{\fX_r} \simeq H_{\fX_r}^{\tilde{g}}$ 
as $H$-representations. 
By \eqref{g1} and Frobenius reciprocity, 
we acquire 
\begin{equation}\label{g22}
 \mathrm{Hom}_{G_1} 
 (\Pi_s,\pi_{\zeta,\chi,c,\omega} ) \simeq 
 \bigoplus_{\overline{H}_r \backslash G_2} 
 \mathrm{Hom}_{H}(H_{\fX_r},\pi_{\zeta,\chi,c,\omega} ). 
\end{equation}

If $\zeta^{\frac{p^m-1}{n_1}} \neq s$, 
the required assertion follows from \eqref{g22} and 
Proposition \ref{pro} \ref{enu:dHHpi}. 
Now, assume that $\zeta^{\frac{p^m-1}{n_1}}=s$. 
Without loss of generality, we may assume that $\zeta$ equals $r$ 
by Lemma \ref{lem:indep}. 
By Proposition \ref{pro} and Frobenius reciprocity,
we obtain a non-zero map
\begin{equation}\label{vap}
 \mathrm{Ind}_{L_r^{\times} U_D^1 \times W_{E_r}}^{\overline{H}_r} 
 ( 
 \theta_{r,\chi,c,\omega} \otimes \tau_{r,\chi,c,\omega}^0 
 ) 
 \to \mathrm{Hom}_{H } (H_{\fX_r},\pi_{r,\chi,c,\omega} ). 
\end{equation}
By applying 
$\mathrm{Ind}_{\overline{H}_r}^{G_2}$ to the map \eqref{vap}, 
we acquire a non-zero map 
\begin{equation}\label{ff}
 \rho_{r,\chi,c,\omega} \otimes \tau_{r,\chi,c,\omega}
 \to \mathrm{Ind}_{\overline{H}_r}^{G_2}\mathrm{Hom}_{H}
 ( H_{\fX_r},\pi_{r,\chi,c,\omega} ).
\end{equation}
We have $\dim \rho_{r,\chi,c,\omega}=(q^n-1)/(q-1)$ and 
$\dim\tau_{r,\chi,c,\omega}=n$. 
Moreover, we have 
\[
 [G_2:\overline{H}_r ]=[F_r : K][D^{\times}: L_r^{\times} U_D^1 ] 
 =\frac{n'(q^n-1)}{n_1(q-1)} 
\]
by the exact sequence 
\[
 1 \to L_r^{\times} U_D^1 \to 
 \overline{H}_r \to W_{F_r} \to 1. 
\]
Hence, the both sides of \eqref{ff} are 
$n(q^n-1)/(q-1)$-dimensional by Proposition \ref{pro} \ref{enu:dHHpi}. 
Since 
$\rho_{r,\chi,c,\omega} \otimes \tau_{r,\chi,c,\omega}$ 
is an irreducible representation of $G_2$, 
we know that \eqref{ff} is an isomorphism 
as $G_2$-representations.
On the other hand, 
we have a non-zero map
\begin{equation}\label{tti}
 \mathrm{Ind}^{G_2}_{\overline{H}_r}\mathrm{Hom}_{H}
 (H_{\fX_r},\pi_{r,\chi,c,\omega} ) 
 \to \mathrm{Hom}_{G_1} (\Pi_s,\pi_{r,\chi,c,\omega} ) 
\end{equation}
induced by a surjective homomorphism 
$\Pi_s|_{H_r} \to H_{\fX_r}$ of $H_r$-representations 
and Frobenius reciprocity. 
Then \eqref{tti} is an isomorphism, 
since the left hand side 
is an irreducible representation of $G_2$ and the 
both sides have the same dimension by \eqref{g22}.
Hence, the required assertion follows from 
the isomorphisms \eqref{ff} and \eqref{tti}. 
\end{proof}

\begin{thm}\label{Pireal}
Let $\mathrm{LJ}$ be the inverse of $\mathrm{JL}$ in Proposition 
\ref{expJL}.
We put 
\[
 \Pi =\bigoplus_{s \in \mu_{\frac{n_1 (q-1)}{p^m -1}}(K)} \Pi_s. 
\] 
Let $\pi$ be a smooth irreducible representation of 
$\iGL_n (K)$. 
Then, we have 
\[
 \mathrm{Hom}_{\mathit{GL}_n(K)} ( \Pi,\pi ) \simeq 
 \begin{cases}
  \mathrm{LJ}(\pi) \otimes \mathrm{LL}(\pi) & 
  \textrm{if $\pi$ is essentially simple supercuspidal,} \\ 
  0 & \textrm{otherwise} 
 \end{cases}
\]
as $D^{\times} \times W_K$-representations.
\end{thm}
\begin{proof}
This follows from Proposition \ref{expJL} 
and Lemma \ref{map}, because 
every essentially simple supercuspidal representation is 
isomorphic to 
$\pi_{\zeta,\chi,c,\omega}$ 
for some 
$\zeta \in \mu_{q-1} (K)$, $\chi \in (k^{\times})^{\vee}$, 
$c \in \overline{\mathbb{Q}}_{\ell}^{\times}$ and 
a smooth character 
$\omega \colon K^{\times} \to \ol{\bQ}_{\ell}^{\times}$.
\end{proof}


\noindent
Naoki Imai\\ 
Graduate School of Mathematical Sciences, 
The University of Tokyo, 3-8-1 Komaba, Meguro-ku, 
Tokyo, 153-8914, Japan\\ 
naoki@ms.u-tokyo.ac.jp \\

\noindent
Takahiro Tsushima\\ 
Department of Mathematics and Informatics, 
Faculty of Science, Chiba University, 
1-33 Yayoi-cho, Inage, Chiba, 263-8522, Japan\\
tsushima@math.s.chiba-u.ac.jp

\end{document}